\documentclass[11pt]{ip-journal}

\usepackage[english]{babel}

\usepackage[letterpaper,top=2cm,bottom=2cm,left=3cm,right=3cm,
marginparwidth=1.75cm]{geometry}
\usepackage[pdftex]{pict2e}
\usepackage[utf8]{inputenc}
\usepackage[T1]{fontenc}
\usepackage{hyperref}
\usepackage{amsmath}
\usepackage{amssymb}
\usepackage{amsfonts}
\usepackage{amssymb}
\usepackage{amsthm}
\usepackage{wrapfig}
\usepackage{amscd}
\usepackage{graphicx}
\usepackage{xcolor}
\usepackage{float}
\usepackage{slashed}
\usepackage{changepage}
\usepackage{caption}

\usepackage{enumerate}

\newtheorem{definition}{Definition}[section]
\newtheorem{lemma}[definition]{Lemma}
\newtheorem{theorem}[definition]{Theorem}
\newtheorem{proposition}[definition]{Proposition}
\newtheorem{cor}[definition]{Corollary}
\theoremstyle{definition}

\theoremstyle{definition}
\newtheorem{remark}[definition]{Remark}
\theoremstyle{definition}

\theoremstyle{definition}
\newtheorem{conjecture}[definition]{Conjecture}

\def\vol{\operatorname{vol}}

\def\tr{\operatorname{tr}}
\def\dist{\operatorname{dist}}
\def\diam{\operatorname{diam}}
\def\minA{\operatorname{minA}}

\newcommand{\D}{\nabla}
\newcommand{\p}{\partial}
\renewcommand{\th}{\theta}

\renewcommand{\bar}{\overline}
\renewcommand{\tilde}{\widetilde}
\newcommand{\RR}{\mathbb{R}}

\DeclareMathOperator{\IN}{IN}
\DeclareMathOperator{\Ree}{Re}
\DeclareMathOperator{\Imm}{Im}

\DeclareMathOperator{\sn}{sn}
\DeclareMathOperator{\cn}{cn}
\DeclareMathOperator{\tn}{tn}
\DeclareMathOperator{\arctn}{arctn}
\renewcommand{\epsilon}{\varepsilon}
\renewcommand{\leq}{\leqslant}
\renewcommand{\geq}{\geqslant}

\newcommand{\G}{{\boldsymbol{g}}}
\newcommand{\g}{g}
\newcommand{\cg}{{\tilde{g}}}
\newcommand{\fg}{{\bar{g}}}

\renewcommand{\D}{{\nabla}}

\newcommand{\cD}{{\tilde{\D}}}
\newcommand{\fD}{{\bar{\D}}}

\newcommand{\dA}{{\,dA}}
\newcommand{\cdA}{{\,d\tilde{A}}}
\newcommand{\fdA}{{\,d\bar{A}}}

\newcommand{\K}{{K}}
\newcommand{\cK}{{\tilde{K}}}

\newcommand{\fDelta}{{\bar{\Delta}}}

\newcommand{\updatetag}[1]{{}}

\title{Drawstrings and flexibility in the Geroch conjecture}
\author{Demetre Kazaras and Kai Xu}

\begin{document}
\maketitle


\begin{abstract}
In this paper, we observe new phenomena related to the structure of 3-manifolds satisfying lower scalar curvature bounds. We construct warped-product manifolds of almost nonnegative scalar curvature that converge to pulled string spaces in the Sormani--Wenger intrinsic flat topology. These examples extend the results of Lee--Naber--Neumayer \cite{LNN} to the case of dimension $3$. As a consequence, we produce the first counterexample to a conjecture of Sormani \cite{SormaniConj} on the stability of the Geroch Conjecture. Our example tests the appropriate hypothesis for a related conjecture of Gromov. On the other hand, we demonstrate a $W^{1,p}$-stability statement ($1\leq p<2$) for the Geroch Conjecture in the class of warped products.
\end{abstract}

\section{Introduction and main results} {\ }

\vspace{-9pt}

\noindent\textbf{1.1. Background.} Rigidity statements are central to our differential geometric understanding of curvature conditions. In the study of scalar curvature lower bounds, fundamental works of Gromov--Lawson \cite{GL} and Schoen--Yau \cite{SY} show the following rigidity statement: {\emph{A Riemannian $n$-torus $(T^n,\G)$ of nonnegative scalar curvature, $R_\G\geq0$, must be flat.}} This is known as the Geroch Conjecture or the Scalar Torus Rigidity Theorem. To deepen our understanding of scalar curvature, the next step is to investigate the associated almost-rigidity or stability problem, which asks the following: \textit{To what extent is a Riemannian torus $(T^n,\G)$ of almost nonnegative scalar curvature $R_\G>-\varepsilon$ close to being flat?} This is a developing subject of active interest -- several special cases have been addressed \cite{Allen_2021, AHLP, Chu-Lee_2022, CKP}, though the general picture remains unclear.

There are known challenges in this problem due to the considerable flexibility of the almost nonnegative scalar curvature condition. One source of this flexibility is the following construction that dates back to Gromov--Lawson \cite{GL} and Schoen--Yau \cite{SY}. Given a Riemannian manifold $(M^n,\G)$ with $R_\G\geq\lambda$, a point $p\in M^n$, and a small $\epsilon>0$. One can locally alter $(M^n,g)$ so that an annular region about $p$ becomes cylindrical, i.e. isometric to the product of a small round $S^{n-1}$ and an interval, and the new scalar curvature is at least $\lambda-\varepsilon$. When $\epsilon\to0$, the cross-sectional width of the cylinder converges to zero. Some detailed constructions can be found in \cite{BDS, Dodziuk, GL, Sweeney_2023}.
Based on this construction, several examples can be made.

\begin{wrapfigure}{r}{0.5\textwidth}
    \centering
    \includegraphics[width=0.48\textwidth]{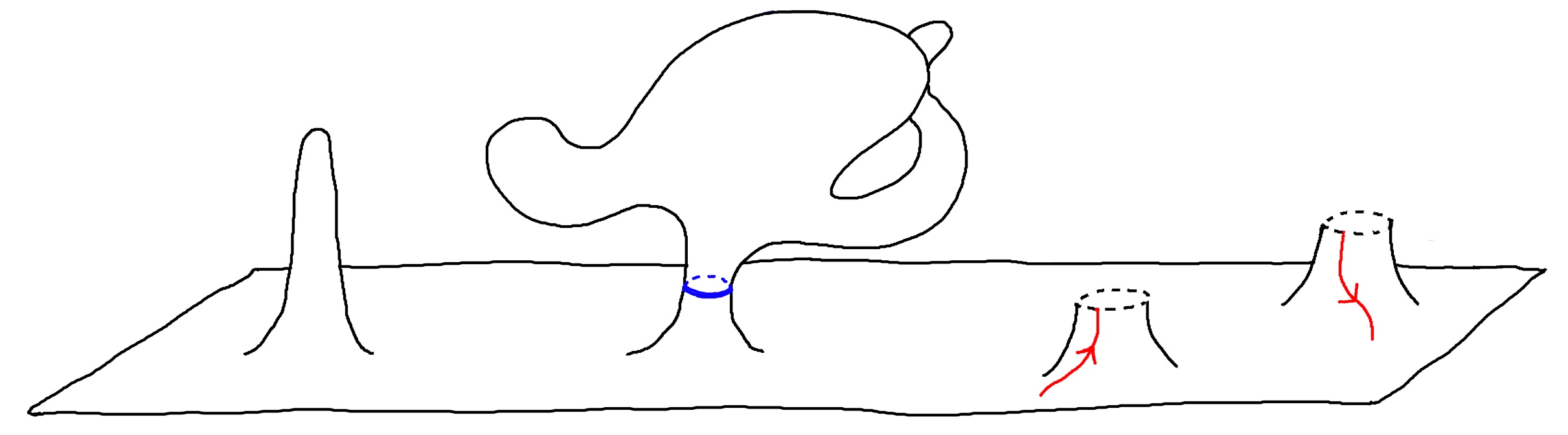}
    \captionsetup{margin=0.8cm}
    \caption{The picture shows, from left to right, the spline, other world, and tunnel examples. 
    }
    \label{fig:tunnels}
    \vspace{-9pt}
\end{wrapfigure}
\begin{enumerate}
    \item If we prolong the cylinder and close it with a spherical cap, then we obtain a \textit{spline}\footnote{There are $C^\infty$-small perturbations of a spline which contain no small minimal surfaces. For this reason, the minA invariant \eqref{e:minA} does not detect spline phenomena.} (or \textit{gravitational well}). Ilmanen observed that by attaching an increasing number of splines with decreasing width to the round sphere, one obtains a Gromov--Hausdorff diverging sequence of positive scalar curvature metrics on $S^n$.

    
    \item Given manifolds $(M^n,\G), (N^n,\mathbf{h})$, we can perform the construction near $p\in M^n$ and $q\in N^n$, and glue the two resulting cylinders. This is often called the \textit{connected sum construction}, which produces metrics on the connected sum $M^n\#N^n$. For $N^n\cong S^n$ this is usually called an \textit{other world} (or \textit{bag of gold} or \textit{bubble}).

    
    \item When one takes $p,q$ lying in a common manifold and applies the previous construction, one creates a \textit{tunnel} connecting $p$ with $q$. If tunnels are built with increasingly dense entries and exits along a given curve in $M$, one may obtain the sewing construction of Basilio--Dodziuk--Sormani \cite{BDS, BS}.
\end{enumerate}

These constructions represent ways in which a lower bound on scalar curvature fails to control metric structures. An almost rigidity statement for the Geroch conjecture needs to either exclude them from the hypotheses, or introduce a notion of convergence that is agnostic to these phenomena. In dimension $3$, Gromov \cite{Billiards} and Sormani \cite{SormaniConj} propose using the \textit{Sormani--Wenger intrinsic flat metric} \cite{Sv2} to measure the distance between tori, and impose a uniform lower bound on the \textit{minA invariant} given by
\begin{equation}\label{e:minA}
    \minA(M,\G):=\inf\{|\Sigma|_{\G}\colon \Sigma\subset M\text{ is a closed embedded minimal surface}\}.
\end{equation}
A minA lower bound forbids extreme other world and tunnel examples which form small minimal surfaces, while the Intrinsic-Flat topology is not significantly impacted by gravity wells. This motivates the following {\emph{minA-IF stability}} conjecture.
\begin{conjecture}\cite[Conjecture 7.1]{SormaniConj}\label{conj:sormanitorus}
    Suppose $\{(M_i,g_i)\}_{i=1}^\infty$ is a sequence of Riemannian $3$-tori with $\diam(M_i,g_i)$ uniformly bounded above and $\vol(M_i,g_i)$ and $\minA(M_i,g_i)$ uniformly bounded away from $0$. If $R_{g_i}\geq-1/i$, then $(M_i,g_i)$ subsequentially converges to a flat torus in the volume-preserving intrinsic flat sense as $i\to\infty$.
\end{conjecture}

\begin{wrapfigure}{r}{0.36\textwidth}
    \centering
    \vspace{-.5cm}
    \includegraphics[width=0.36\textwidth]{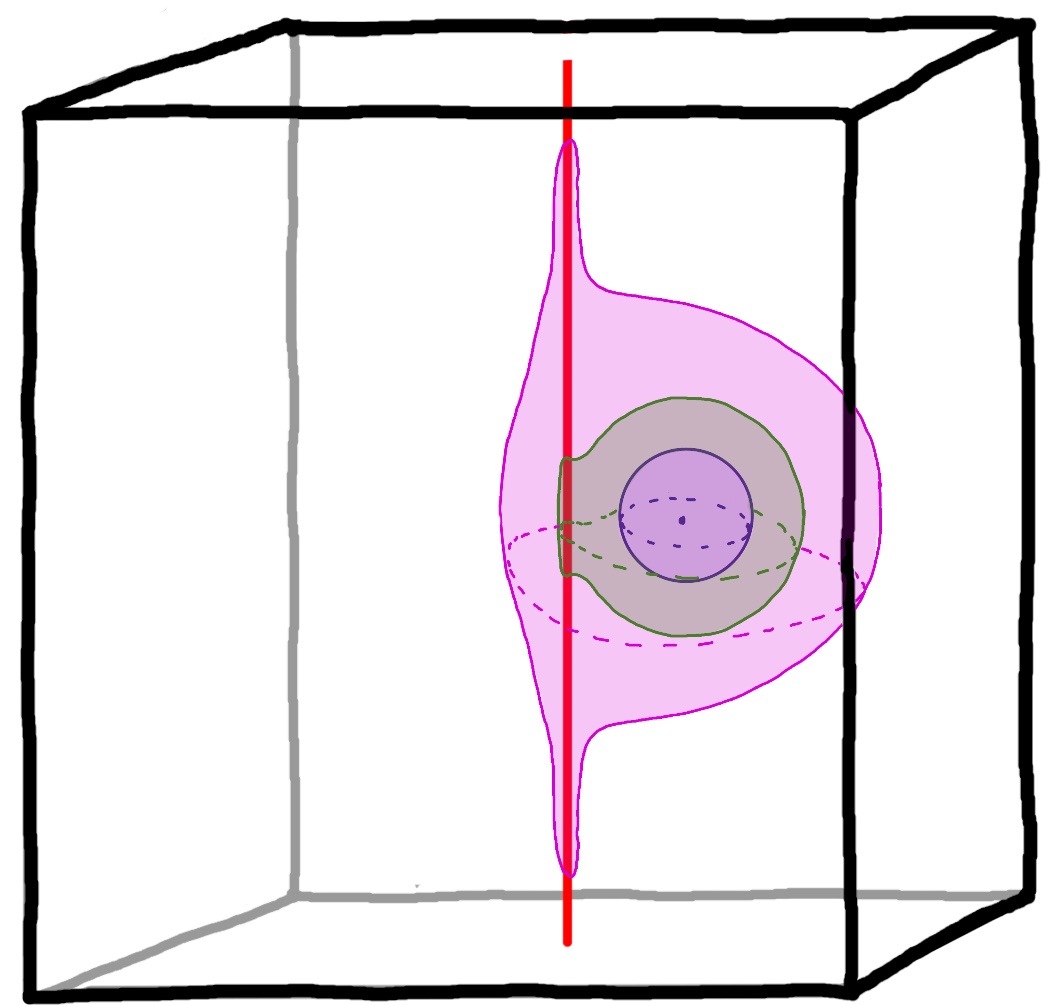}
    \captionsetup{margin=0.2cm}
    \caption{Concentric metric balls in $(M_i,\G_i)$ of Theorem \ref{t:T3}.
    }
    \label{fig:dsball}
\end{wrapfigure}

\noindent\textbf{1.2. Drawstrings.} In addition to the examples mentioned above, we discover a new flexibility phenomenon in scalar curvature, which we call {\emph{drawstrings}}. Given a flat 3-torus, the drawstring construction alters the Riemannian metric near a closed geodesic, reducing its length to less than $\epsilon$ while keeping scalar curvature $R\geq-\epsilon$. This geodesic is a metaphorical drawstring, scrunching the fabric of space and shortening the distance between otherwise disparate points. 
A metric modeling this behavior is a warped product of a surface and a line $\G=g+\varphi^2dt^2$, where the function $\varphi$ equals 1 outside an $\epsilon$-ball and rapidly decreases to $\epsilon$ at its origin. As $\epsilon\to0$, our examples converge to a \textit{pulled string space}, where the given geodesic has zero length while the remainder of the torus is flat. Moreover, the drawstring construction satisfies a uniform minA lower bound. Thus, we are led to counterexamples of Conjecture \ref{conj:sormanitorus}.

\vspace{6pt}

\begin{theorem}\label{t:T3}
There exists a sequence of $3$-dimensional Riemannian tori $\{(M_i,\G_i)\}_{i=1}^\infty$ and constants $D_0,V_0,A_0>0$ satisfying the following for all $i$:
\begin{align}
    \begin{split}
        &\quad R_{\G_i}\geq-\frac1i,\quad D_0^{-1}\leq \mathrm{diam}(M_i,\G_i)\leq D_0,\\
        &V_0^{-1}\leq \vol(M_i,\G_i)\leq V_0,\quad \minA(M_i,\G_i)\geq A_0,
    \end{split}
\end{align}
such that $\{(M_i,\G_i)\}_{i=1}^\infty$ converges in the Gromov--Hausdorff and volume-preserving intrinsic flat senses to a pulled-string space. The limit space is obtained from a flat $3$-torus by pulling a circle to a point, see Section \ref{sec:scrunch}. In particular, the limit space is not isometric to a flat torus.
\end{theorem}

Pulled string limit spaces were previously discovered by Basilio--Dodziuk--Sormani \cite{BDS} and Basilio--Sormani \cite{BS} in the context of scalar curvature lower bounds. Their examples are built with the sewing construction, where the convergence to pulled string limit spaces is due to increasingly dense shortcuts. The sewing construction leverages small tunnels, and so the examples in \cite{BDS, BS} have minA invariant converging to zero.

The present work is partially inspired by Lee--Naber--Neumayer \cite{LNN}, where analogous drawstring phenomena were discovered in dimensions 4 and above, and lead to extreme examples (see also Lee-Topping \cite{Lee-Topping}). We emphasize that previous examples are centered on codimension 3 submanifolds. The Gromov--Lawson and Schoen--Yau constructions form cylinders about points in manifolds of dimension at least 3, and the Lee--Naber--Neumayer construction forms degenerations near curves in manifolds of dimension at least 4. On the other hand, the constructions of Theorem \ref{t:T3} appear codimension 2 in nature, see the following subsection for more details.


\vspace{12pt}

The drawstring construction can also be carried out in the context of uniformly positive scalar curvature. As a result, we produce new extreme examples that probe the {\emph{minA Scalar Compactness Conjecture}} \cite[Conjecture 4.1]{SormaniConj}. This conjecture posits that $3$-manifolds of \textit{positive} scalar curvature and uniformly positive $\minA$ invariant are precompact in the volume-preserving intrinsic flat topology, and the limit spaces satisfy some weak notion of positive scalar curvature. The following result may assist in the search for the appropriate weak notion.

\begin{theorem}\label{t:S2S1}
    There exists a sequence $\{(M_i,\G_i)\}_{i=1}^\infty$ of Riemannian $S^2\times S^1$ and constants $A_0,D_0,V_0>0$ satisfying the following for all $i$:
    \begin{align}
    \begin{split}
        & \;\;\;R_{\G_i}\geq2-\frac1i,\quad D_0^{-1}\leq \mathrm{diam}(M_i,\G_i)\leq D_0,\\
        &V_0^{-1}\leq \vol(M_i,\G_i)\leq V_0,\quad \minA(M_i,\G_i)\geq A_0,
    \end{split}
\end{align}
so that $\{(M_i,\G_i)\}_{i=1}^\infty$ converges to a pulled-string space in the Gromov--Hausdorff and volume-preserving intrinsic flat senses. The limit space is obtained from a product $S^2\times S^1$ pulled along a circle $\{*\}\times S^1$. In particular, the limit space has non-Euclidean tangent cones and does not satisfy the volume-limit notion of positive scalar curvature.\footnote{We say that a metric space $(X,d)$ has non-Euclidean tangent cones if for some $p\in X$, the rescaled spaces $(X,\tfrac{1}{r}d,p)$ do not converge to $\RR^n$ as $r\to0$. A metric space satisfies the volume-limit notion of positive scalar curvature if $\lim_{r\to0}r^{-n-2}(|B(x,r)|-\omega_nr^n)\leq0$ for all points $x$ where $\omega_n$ is the volume of the unit ball in $\mathbb{R}^n$.}
\end{theorem}

\vspace{6pt}

\noindent\textbf{1.3. Further geometry of drawstrings.} We now give a more detailed geometric description of the drawstring examples. The metrics in Theorem \ref{t:T3} take the warped product form $\G=\g+\varphi^2dt^2$ where $\g$ and $\varphi$ are a metric and function on a $2$-torus. The warping factor $\varphi$ equals 1 outside the disc $B^{\g}(x_0,\epsilon)$ and decreases to $\epsilon$ at $x_0$, where $x_0$ is a basepoint. The central challenge is to achieve this task while keeping the scalar curvature almost nonnegative. In light of the formula $R_{\G}=2(
K_g-\varphi^{-1}\Delta_g\varphi)$, we must create sufficiently large Gauss curvature for the base metric.\footnote{For example, it is not possible to form a drawstring with flat base metric $\g$: otherwise the warping factor satisfies $\Delta_g\varphi\leq\epsilon\varphi$, and Moser's Harnack inequality poses a strong constraint on the minimum of $\varphi$.} In the higher dimensional constructions of Lee--Naber--Neumayer \cite{LNN}, the $(n-1)$-dimensional base metric forms a slightly acute cone near $x_0$. When $n\geq4$ the cone-like region contributes tremendously positive curvature $O(1/r^2)$ for small radius $r$, that turns out to be sufficient for decreasing $\varphi$ to almost zero.

In the present $3$-dimensional work, the above approach does not work directly since $2$-dimensional cones are flat. On the other hand, the tip of an acute $2$-dimensional cone still carries positive curvature in the distributional sense. Our construction is done by properly smoothing the cone vertex, essentially reflecting that the curvature concentrated there is large enough to allow rapid decrease of $\varphi$. The smoothing is modeled on the following metric on $D^2\times\mathbb{R}$:
\begin{align}\label{eq-intro:metric}
    \G=\varphi(r)^{-2}\big(dr^2+f(r)^2d\th^2\big)+\varphi(r)^2dt^2,
\end{align}
where $(r,\theta)$ are polar coordinates on the disc, with 
\begin{align}\label{eq-intro:f_and_u}
     f(r)=r\left(1-\frac{c_1}{\log(1/r)}\right), \qquad \varphi(r)=\left(\log\frac1r\right)^{-c_2},
\end{align}
where $c_1,c_2$ are small positive constants to be carefully chosen. The scalar curvature of such metric is explicitly computed to be 
\begin{equation}\label{eq-intro:scalar_expression}
    R_{\G}=\frac2{r^2(\log 1/r)^{2+2c_2}}\Big[\frac{c_1(c_1+2)}{\log(1/r)-c_1}+c_1-c_2^2\Big],
\end{equation}
which is positive for small $r$ when $c_1\geq c_2^2$. 
The horizontal disks $D^2\times\{t\}$ of \eqref{eq-intro:metric} represent the smoothed cones and possess Gauss curvature on the order of $O\big(1/(r^2|\log(r)|^{2+2c_2})\big)$.

The key technical result, Theorem \ref{t:T3tube}, locally modifies the product of a constant curvature surface $S^2_k$ with a line to create a drawstring, while keeping scalar curvature at least $2(k-\epsilon)$. Two independent proofs of this result are given. While the two constructions are technically different, they both make use of the metric \eqref{eq-intro:metric} and are ultimately similar in spirit.

The first proof of Theorem \ref{t:T3tube}, given in Section \ref{sec:epsilonpipe}, is geometric and reflects the heuristic of smoothing cones that we mentioned above. The process takes place on a cylindrical region about a central line in the product $S^2_k\times\mathbb{R}$. We remove a small cylindrical neighborhood of a central line and glue in a metric $\G_1=g'+dt^2$, where $g'$ is a metric on the disk that has constant curvature $k-\epsilon$ and forms a slightly acute cone at the origin. With appropriate choices of the parameters $\epsilon, c_1,c_2$, and the cone angle, one can further remove a smaller central region in $\G_1$ and glue in the piece \eqref{eq-intro:metric} in a $C^{1,1}$ manner. The singularities in the resulting metric are mild and can be resolved. We remark that the slightly acute cone is the only region where the scalar curvature is less than $2k$. More details on this construction can be found below Theorem \ref{t:T3tube}. The second proof of Theorem \ref{t:T3tube} is presented in Section \ref{sec:analytic}, where a single explicit formula for the drawstring metric is given and further smoothing is not required. This formula directly concatenates the metric \eqref{eq-intro:metric} to the product metric via cutoff functions. While the second proof is computationally simpler, it is geometrically less transparent.

\begin{remark}\label{rem:ricci}
    The drawstring phenomenon is unique to scalar curvature, and is not possible in the setting of Ricci curvature lower bounds. Roughly speaking, if one creates a drawstring in the $\epsilon$-neighborhood of a closed geodesic $\gamma$, then for a point $x\in\gamma$ we have $|B(x,2\epsilon)|\geq C\epsilon^{n-1}$. This would violate the volume comparison theorem if the Ricci curvature is uniformly bounded below when $\epsilon\to0$.
\end{remark}


\noindent\textbf{1.4. A $W^{1,p}$-stability statement.} Given that the minA-IF torus stability conjecture fails, it is natural to wonder if one may replace or enhance either the minA lower bound or the intrinsic flat topology in order to formulate an improved conjecture. While this task is essential, the appropriate modification unfortunately remains unclear to the present authors. On the other hand, within the class of warped products we are able to describe a natural and sharp version of Scalar Torus Stability. More specifically, we prove convergence to a flat $3$-torus in $W^{1,p}$-topology for $p\in[1,2)$.

\begin{theorem}\label{thm-intro:convergence_T3}
    Let $\{(M_i,\G_i)\}_{i=1}^\infty$ be a sequence of Riemannian $3$-tori of the form $\G_i=g_i+\varphi_i^2dt^2$, where $g_i$ and $\varphi_i>0$ are metrics and functions on $2$-tori, and $t$ parameterizes a vertical circle whose length may depend on $i$. Assume the following hold for all $i$:
    \begin{equation}\label{eq-intro:convergence_T31}
        R_{\G_i}\geq-2/i,\quad \diam(M_i,\G_i)\leq D_0,\quad \vol(M_i,\G_i)\leq V_0,\quad \minA(M_i,\G_i)\geq A_0.
    \end{equation}
    Then there is a flat torus $(T^3,\G_\infty)$ so that a subsequence of $\G_i$ converges to $\G_\infty$ in the $W^{1,p}$-sense for any $1\leq p<2$. Specifically, there are diffeomorphisms $\Phi_i:(T^3,\G_\infty)\to(M_i,\G_i)$ such that
    \begin{equation}\label{eq-intro:convergence_T32}
        ||\G_\infty-\Phi_i^*\G_i||_{W^{1,p}(\G_\infty)}\to0\ \ \text{as}\ \ i\to\infty,\ \ \text{ for  }\ \ 1\leq p<2.
    \end{equation}
    In particular, for all $q\geq1$ we have $||\G_\infty-\Phi_i^*\G_i||_{L^q(\G_\infty)}\to0$ as $i\to\infty$.
\end{theorem}

\begin{remark}
    As a consequence of the warped product form of the drawstring examples and Theorem \ref{thm-intro:convergence_T3}, the manifolds in Theorem \ref{t:T3} converge to a flat torus in the $W^{1,p}$-sense, for $1\leq p<2$.
\end{remark}

The diffeomorphisms in Theorem \ref{thm-intro:convergence_T3} are built by uniformizing the underlying $2$-tori of $(M_i,\G_i)$. A primary step in the proof is to establish that the corresponding flat $2$-tori lie within a compact subset in the moduli space of flat tori. Next, the Moser-Trudinger and Brezis-Merle inequalities are employed to show that the warping factors and the conformal factors (that arise from uniformization) are approaching constant values in $W^{1,p}$-norm when the scalar curvature is almost nonnegative. The minA lower bound is used to control the range of the $t$-coordinate, as well as to prevent the collapsing of the flat $2$-tori. We remark that the restriction $p<2$ in \eqref{eq-intro:convergence_T32} is expected to be sharp. This is due to an example discovered by Sormani--Tian--Wang \cite{STW23}, where the scalar curvature compactness conjecture is analyzed in the class of warped products $S^2\times S^1$ over the round $2$-sphere. In that construction, the warping factors resemble a Green's function and therefore diverge in $W^{1,2}$-norm. 

Theorem \ref{thm-intro:convergence_T3} is similar in spirit to a number of previous works analysing warped product manifolds. In \cite[Theorem 1.4]{Bryden}, Bryden obtains a $W^{1,p}$-convergence, $p\in[1,2)$, to flat $3$-space for a class of axially-symmetric asymptotically Euclidean manifolds with non-negative scalar curvature and ADM masses tending to $0$. Sequences of $3$-dimensional warped products over a fixed flat $2$-torus are studied by Allen-Hernandez-Vazquez-Parise-Payne-Wang in \cite{AHLP}, and the minA-IF torus stability conjecture is verified in this context. Tian-Wang \cite{TW} study the minA-IF scalar compactness conjecture for $3$-dimensional warped products over the round unit $2$-sphere, showing an $L^q$-convergence statement for all $q<\infty$. In these works the underlying $2$-dimensional geometry is fixed, thereby excluding the phenomena of Theorems \ref{t:T3} and \ref{t:S2S1}.

The minA-IF torus stability conjecture has also been studied outside the context of warped products. In \cite{Allen_2021} Allen considers metrics conformal to flat tori (in arbitrary dimensions) which satisfy a type of no-bubbling condition, confirming a special case of the stability conjecture. In the article \cite{Chu-Lee_2021}, Chu-Lee show a torus stability result for Riemannian tori which are conformally related by a uniformly $L^p$ factor to metrics of bounded geometry. Using the Yamabe flow, they obtain volume preserving intrinsic flat convergence to a flat torus in all dimensions. In a separate work, Chu-Lee analyzed the Ricci flow on metrics that satisfy pointwise bounds above and below by a background metric in the context of uniform $W^{1,p}$-bounds, proving an $L^p$ convergence of the metric tensors \cite[Theorem 5.1]{Chu-Lee_2022}. The $\minA$-IF torus stability conjecture has also considered the case where the almost non-negative scalar curvature tori arise as graphs over flat tori. In \cite[Theorem 1.4]{CKP}, Cabrera Pacheco-Ketterer-Perales establishes a torus stability result for a special class of such tori. \\

This paper is organized as follows. The short preliminary Section \ref{sec:preliminary} introduces pulled-string spaces and a criterion due to Basilio--Dodziuk--Sormani for convergence to such spaces. The main technical constructions behind our drawstring examples are given in Sections \ref{sec:epsilonpipe} and \ref{sec:analytic} where two separate proofs of Theorem \ref{t:T3tube} are given. In Section \ref{sec:mainthm} we prove Theorems \ref{t:T3} and \ref{t:S2S1}. The proof of Theorem \ref{thm-intro:convergence_T3} is given in Section \ref{sec:lp}. Finally, Appendix \ref{sec:L1_est} contains some estimates that are necessary for the proof of Theorem \ref{thm-intro:convergence_T3}.

\vspace{12pt}

\textbf{Acknowledgements.} The authors would like to thank Professors Man-Chun Lee, Blaine Lawson, Brian Allen, and Edward Bryden for insightful discussions. We owe a debt of gratitude to Professor Christina Sormani for her inspiring contributions to this area, general encouragement, and remarkable support for young mathematicians. We also thank Christina Sormani for suggesting we find a more computationally transparent proof, leading us to the analytic proof presented in Section \ref{sec:analytic}.









\section{Preliminaries}\label{sec:preliminary}

{\bf{Notation:}} Given a subset $U$ of a metric space $(X,d^X)$, we use $B^X(Y,r)$ (or simply $B(Y,r)$ when the ambient metric space is clear in context) to denote $r$-distance neighborhood $B^X(Y,r)=\{x\in X: d^X(x,Y)<r\}$. If $(X,d^X)$ arises from a Riemannian manifold $(M,\G)$, we use the notation $B^{\G}(Y,r)$ for distance neighborhoods. We always use boldface symbol $\G$ to denote metrics on 3-manifolds, and use plain symbols $g$ (as well as $\cg,\fg$ which are used in Section \ref{sec:lp}) for metrics on surfaces.

\subsection{Integral current spaces and the Sormani--Wenger intrinsic flat metric}

In \cite{AK} Ambrosio-Kirchheim introduced the notion of {integral currents} on metric spaces, which Sormani--Wenger \cite{Sv2} used to define {\emph{integral current spaces}}. Such spaces are special metric spaces equipped with extra data allowing them to perform calculations that generalize the integration of $n$-forms in a manifold. Sormani--Wenger constructed a natural topology on the space of integral current spaces, induced by a metric called the {\emph{Sormani--Wenger intrinsic flat}} or simply {\emph{intrinsic flat}} distance. The present work will not require any technicalities of these constructions, but the interested reader is referred to \cite{Sv2}.

\subsection{Scrunching and pulling threads in metric spaces}\label{sec:scrunch}

Suppose we are given a metric space $(X,d^X)$ and a compact curve $\sigma\subset X$. One may use this data to construct a {\emph{pulled string}} space $(Y,d^Y)$ by, informally speaking, declaring $\sigma$ to have vanishing length. Precisely, the metric space $(Y,d^Y)$ is given by
\begin{equation}
    Y=(X\setminus \sigma)\cup\{\sigma\},\quad d^Y(x,y)=    \min\left\{d^X(x,y),d^X(x,\sigma)+d^X(\sigma,y)\right\}.
\end{equation}
We refer to $(Y,d^Y)$ as {\emph{created from $(X,d^X)$ by pulling $\sigma$ to a point}}. In the case where $(X,d^X)$ carries an integral current structure, such as the case of a Riemannian manifold, an associated structure is induced on the pulled-string space, see \cite{BDS} for details.

The next task is to articulate a principle that allows us to estimate the intrinsic flat distance between our constructions and pulled-string spaces. The following definition describes, in robust terms, situations where a string is being drawn to a point over a sequence of spaces.
\begin{definition}\cite{BDS,BS}\label{def:scrunch} {\ }
    Let $(M,\G)$ be an oriented closed Riemannian manifold with a given compactly embedded curve $\sigma\subset M$. A sequence of oriented closed Riemannian manifolds $\{(N_i,\G_i)\}_{i=1}^\infty$ is said to {\emph{scrunch $\sigma$ down to a point}} if there exists compact subsets $U_i\subset N_i$ and numbers $\varepsilon_i,\delta_i,H_i>0$ which satisfy the following for all $i=1,2,\dots$:
    \begin{enumerate}[(i)]
        \item $(N_i\setminus U_i,\G_i)$ is isometric to $(M\setminus B^\G(\sigma,\delta_i),\G)$,
        \item $\vol(U_i,\G_i)\leq\vol\left(B^\G(\sigma,\delta_i),\G\right)(1+\varepsilon_i)$ and $\vol(N_i,\G_i)\leq\vol\left(M,\G\right)(1+\varepsilon_i)$,
        \item $\diam(U_i,\G_i)\leq H_i,$\footnote{In \cite{BDS,BS}, there is an extra stipulation that $2\delta_i<H_i$. If one has a sequence scrunching a curve $\sigma$ to a point in the sense of Definition \ref{def:scrunch}, one may achieve this extra condition by re-choosing $H_i$ to be $\max\{3\delta_i,\diam(U_i,\G_i)\}$.}
    \end{enumerate}
    with $\varepsilon_i,\delta_i,H_i\to0$ as $i\to\infty$.
\end{definition}

It is a fundamental observation of Basilio--Dodziuk--Sormani \cite{BDS} that sequences which scrunch curves to a point must converge to a pulled-string space.

\begin{theorem}\cite[Theorem 2.5]{BS}\label{t:scrunch}
 Let $(M,\G)$ be an oriented closed Riemannian manifold with a given compactly embedded curve $\sigma\subset M$. If a sequence of oriented closed Riemannian manifolds $\{(N_i,\G_i)\}_{i=1}^\infty$ scrunches $\sigma$ to a point, then $(N_i,\G_i)$ converges to the space created from $(M,\G)$ by pulling $\sigma$ to a point in both the Gromov--Hausdorff and intrinsic flat senses.
\end{theorem}

\subsection{The curvature of some warped products}
To conclude this section, we describe properties of the class of $3$-dimensional Riemannian metrics studied throughout the paper. To fix the setting, suppose we are given a number $r_*$ and two smooth functions $f,u:[r_*,\infty )\to\mathbb{R}$ such that $f>0$. With this data, one may consider the $3$-dimensional Riemannian metric
\begin{equation}\label{eq-geometry:mainwarpedprod}
    \G=e^{-2u(r)}\left(dr^2+f(r)^2d\theta^2\right)+e^{2u(r)}dt^2.
\end{equation}
We take this opportunity to record an elementary but essential computation that we make frequent use of. Notice that the special form \eqref{eq-geometry:mainwarpedprod} ensures that the second derivative of $u$ does not appear in the scalar curvature formula.

\begin{proposition}\label{prop:curvatures}

The scalar curvature of $\G$ given in \eqref{eq-geometry:mainwarpedprod} is 
\begin{equation}\label{e:scalg}
    R_{\G}=2e^{2u}\Big(-\frac{f''}f-(u')^2\Big),
\end{equation}
where $'$ denotes $\tfrac{d}{dr}$, and
the mean curvatures of constant $r$-coordinate surfaces in the positive-$r$ direction are given by 
\begin{equation}\label{e:meang}
    H=e^u\frac{f'}{f}.
\end{equation}
\end{proposition}
\begin{proof}
Denote $g=e^{-2u}(dr^2+f^2d\theta^2)$ and $\tilde g=dr^2+f^2d\theta^2$, so $g=e^{-2u}\tilde g$. It is known \cite[Proposition 1.13]{geometric} that the scalar curvature of metrics of the form $g+\varphi^2\,dt^2$ (where $\varphi$ is independent of $t$) is $2(K_g-\frac{\Delta_g\varphi}{\varphi})$. Taking $\varphi=e^u$, we have
\begin{align}
    \frac12R_{\G} &= K_g-\frac{\Delta_g e^u}{e^u} \\
    &= e^{2u}(K_{\tilde g}+\Delta_{\tilde g}u)-\frac{e^{2u}\cdot e^u(\Delta_{\tilde g}u+|\D_{\tilde g}u|_{\tilde g}^2)}{e^u} \\
    &= e^{2u}\left(K_{\tilde g}-|\D_{\tilde g}u|_{\tilde g}^2\right) \\
    &= e^{2u}\left(-\frac{f''}f-(u')^2\right).
\end{align}
Use $\Sigma_r$ to denote the surfaces of constant $r>0$ coordinate. Note that the unit normal vector of $\Sigma_r$ in the positive $r$ direction is $e^u\p_r$. The mean curvature of $\Sigma_r$ is computed as
\begin{equation*}
    H=\frac12 e^u\tr_{\Sigma_r}\left[\partial_r(e^{-2u}f^2)d\theta^2+\partial_r e^{2u}dt^2\right]=
    e^u\left(\big(\frac{f'}f-u'\big)+u'\right)=e^u\frac{f'}f. \qedhere
\end{equation*}
\end{proof}


\section{Construction by smoothing a cone}\label{sec:epsilonpipe}

This section provides the construction used to prove Theorems \ref{t:T3} and \ref{t:S2S1}. This geometry and its properties are found in Theorem \ref{t:T3tube} below. Given a number $k\in\RR$, we write $S^2_k$ for the simply-connected space form of constant Gauss curvature $k$. We use $D^2$ to denote a 2-dimensional open disk.
\begin{theorem}\label{t:T3tube}
    Given $k\in\mathbb{R}$ and any small numbers $\varepsilon,\delta,r_0>0$, there exists a smooth doubly warped product metric
	\begin{equation}\label{eq-pipe:general_form}
		\G=e^{-2u(r)}\big(dr^2+f(r)^2d\th^2\big)+e^{2u(r)}dt^2
	\end {equation}
    on the solid cylinder $D^2\times\RR$ in cylindrical coordinates $\{r,\theta,t\}$, which satisfies the following properties:
    \begin{enumerate}[(I)]
        \item the scalar curvature of $\G$ is at least $2(k-\varepsilon)$,
        \item a neighborhood of $\p D^2\times\RR$ is isometric to the annular region $\big\{\frac32r_1\leq d_{S^2_k}(-,x_0)\leq 2r_1\big\}\times\RR\subset S^2_k\times\RR$ for some $r_1>0$ satisfying $r_1\leq r_0$ (where $x_0\in S^2_k$ is any basepoint),
	\item $u=0$ in a neighborhood of $\p D^2\times\RR$,
        \item $e^{u}\leq\delta$ in a neighborhood of $r=0$,
        \item the surfaces $\Sigma_t=\{r=t\}\times\RR$ are strictly mean convex,
	\item $d_{\G}(\{0\}\times \RR,\partial D^2\times \mathbb{R})<r_0$.
        \item $\vol(D^2\times[0,1],\G)<100V_0$, where $V_0$ is the volume of $\{0\leq d(-,x_0)\leq 2r_1\}\times[0,1]\subset S^2_k\times[0,1]$.
    \end{enumerate}
\end{theorem}

\begin{remark}\label{rem:epsilonpipe}
    Notice that the quantity $V_0$ in property (VII) is $2\pi(1-\cn_k(2r_1))/k$ when $k\neq0$ and $4\pi r_1^2$ when $k=0$; see below \eqref{eq-pipe:def_sn_k} for the definition of $\cn_k$. In either case, $V_0=O(r_1^2)$ for a fixed $k$.
\end{remark}

As mentioned above Remark \ref{rem:ricci}, two independent proofs of Theorem \ref{t:T3tube} are given. We call them Proofs A and B. Proof A comprises the remainder of the section.

\begin{proof}[Proof A outline:]
Before the full proof begins in the next subsection, we present some preliminary notations and an overview of the proof strategy. 

The metric $\G$ is obtained by gluing four pieces consecutively. The metric on each piece takes the form
\begin{equation}
    \G_i=e^{-2u_i(r)}(dr^2+f_i(r)^2d\th^2)+e^{2u_i(r)}dt^2,\quad r\in[a_i,b_i],\quad  t\in\RR,
\end{equation}
and they are glued together by identifying $\G_i|_{\{r=a_i\}}$ with $\G_{i+1}|_{\{r=b_{i+1}\}}$, for $i=1,2,3$. The four pieces have the following roles.
\begin{enumerate}
\item The outermost piece $\G_1$ is an annular region centered along a line $\{*\}\times\mathbb{R}$ in the manifold $S^2_k\times\RR$, where $S^2_k$ is the simply connected Riemannian surface of constant curvature $k$. This piece ensures condition $(II)$ of Theorem \ref{t:T3tube}.

\item The second piece $\G_2$, glued to the inner boundary of $\G_1$, is an annular region in the manifold $S^2_{\alpha,k-\epsilon}\times\RR$ where $S^2_{\alpha,k-\epsilon}$ is the simply connected Riemannian surface of constant curvature $k-\epsilon$ and one conic singularity of angle $2\pi\alpha$ for $0<1-\alpha\ll1$. The nontrivial cone angle $\alpha<1$ is crucial for the construction since it creates positive curvature at the tip. On the other hand, to avoid creating large negative curvature while gluing with the first piece, we need to slightly decrease its curvature to $k-\epsilon$.

\item The third piece $\G_3$ is the core component where the large decrease in the $dt$ line element occurs. Here $f_3$ is the prototype function $r(1-\frac{c_1}{\log(1/r)})$ and $u_3$ is a modification of the function $-c_2\log\log(1/r)$. These prototype functions are equivalent to the ones introduced in \eqref{eq-intro:f_and_u}, by the identification $u_3=\log\phi$. In the construction, the modification for $u_3$ is done to appropriately glue with the nearby pieces.

\item The metric $\G_3$ is singular at $r=0$. To remove this singularity, we cut off $\G_3$ at a certain small radius and glue in an additional piece $\G_4$. This metric takes the form of the product of $\RR$ with a disk with sufficiently large constant curvature.
\end{enumerate}
\noindent See Figure \ref{fig:graph1} for a schematic depiction of this construction.

Roughly speaking, the functions $f_i$ and $u_i$ are concatenated to form the functions $f$ and $u$ in Theorem \ref{t:T3tube}. The function $u$ is constructed to be smooth. The function $f$ is constructed to be piecewise smooth, but merely $C^{1,1}$ where one transitions from one piece to the next. A final step is needed to mollify $f$. This last step and the modification mentioned in piece 3 are accomplished by the use of two technical statements, Lemmas \ref{lemma-pipe:smoothing_u} and \ref{lemma-pipe:smoothing_f}.

To explain the use of these lemmas, recall the following well-known principle for gluing two manifolds along a common boundary: If two Riemannian manifolds $(M_1,h_1),(M_2,h_2)$ have isometric boundary and their outward boundary mean curvatures satisfy $H_{M_1}+H_{M_2}\geq0$, then one can mollify the metric $h_1\sqcup h_2$ on the union to produce a smooth metric with scalar curvature at least $\min\{R_{h_1},R_{h_2}\}-\varepsilon$. See for instance \cite[Theorem 5]{BMN}. The smoothing lemmas used in the present construction play a role similar to the above gluing statement, but have the advantage of preserving the particular warped product form \eqref{eq-pipe:general_form}.

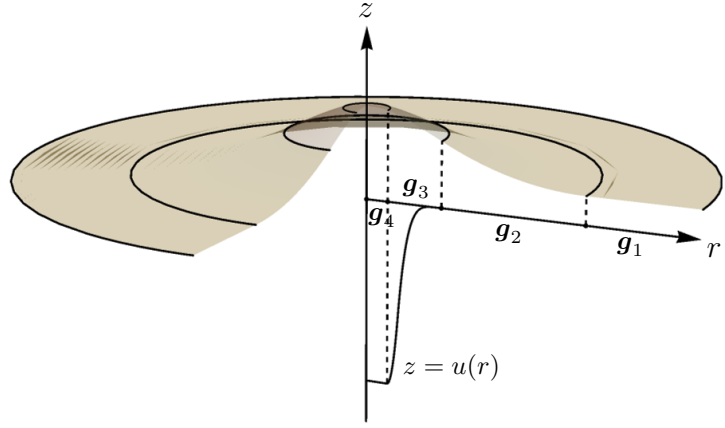
\begin{figure}[h]
    \begin{center}
        \includegraphics[scale=0.5]{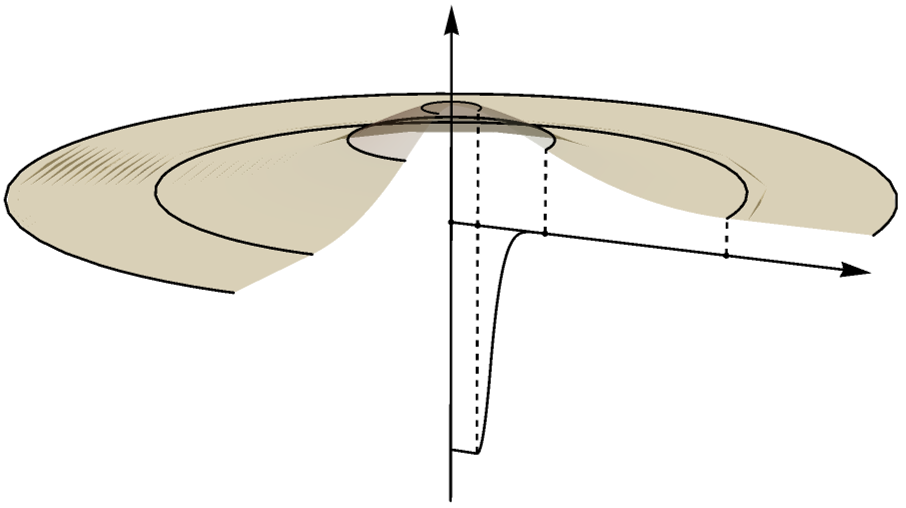}
    \end{center}
    \begin{picture}(0,0)(221.5,1)
        \put(316,87){$\G_1$}
        \put(270,92){$\G_2$}
        \put(235,109){$\G_3$}
        \put(222,98){$\G_4$}
        \put(218,175){{\large{$z$}}}
        \put(235,40){$z=u(r)$}
        \put(350,84){{\large{$r$}}}
    \end{picture}
    \captionsetup{margin=.8cm}
    \caption{The $2$-dimensional metric $e^{-2u}(dr^2+f^2d\theta^2)$ of Theorem \ref{t:T3tube} is realized isometrically as the graph of a radial function over $\mathbb{R}^2$, displayed in grey. Below it, the graph of $u$ is given. The vertical dashed lines indicate the junctions between neighboring pieces of the construction.}\label{fig:graph1}
\end{figure}

More notation is needed. For $k\in\RR$, we define the functions
\begin{equation}\label{eq-pipe:def_sn_k}
    \sn_k(r)=\begin{cases}
    \frac1{\sqrt k}\sin\big(\sqrt kr\big)&\text{ for }k>0, \\
    r&\text{ for }k=0, \\
    \frac1{\sqrt{-k}}\sinh\big(\sqrt{-k}r\big)&\text{ for }k<0,
    \end{cases}
\end{equation}
so the metric of constant curvature $k$ is expressed as $dr^2+\sn_k(r)^2d\theta^2$ in polar coordinates. We also denote $\cn_k(r)=\sn'_k(r)$, $\tn_k(r)=\sn_k(r)/\cn_k(r)$, and $\arctn_k(r)$ as the inverse function of $\tn_k(r)$. Thus $k\sn_k^2+\cn_k^2=1$ and $\cn'_k=-k\sn_k$.

The proof of Theorem \ref{t:T3tube} is broken into two parts and contained in the two subsections below. The first part contains the detailed construction of the metrics $\G_1, \dots, \G_4$ described above, as well as some of their relevant properties. In the second part, we show that the resulting metric can be smoothed and that it satisfies conditions (I)\,--\,(VII).

\subsection{Proof A step 1: construction of the four pieces}\label{subsec:part1}
\vspace{6pt}
We now begin the first formal proof of Theorem \ref{t:T3tube}. Since decreasing $\epsilon$ and $r_0$ only strengthens properties (I)\,--\,(VII), we may assume that
\begin{equation}\label{eq-pipe:assumptions}
\epsilon\leq 10^{-3}\min\{2,1+|k|\},\quad r_0\leq 10^{-3}\min\left\{1,|k|^{-1}\right\}.
\end{equation}
Having assumed \eqref{eq-pipe:assumptions}, we fix the parameters $\epsilon,r_0$. \\


\noindent\textbf{Piece (i)} is a cylindrical region in the model space $S^2_k\times\RR$:
\begin{equation}\label{eq-pipe:def_g1}
    \G_1=dr^2+\sn_k^2(r)d\th^2+dt^2,\quad r_1\leq r\leq 2r_1,\ t\in\RR.
\end{equation}
The warping factors for this piece are thus $f_1(r)=\sn_k(r)$, $u_1\equiv0$. 

The radius $r_1$ must be chosen to satisfy various smallness conditions. To clarify this choice and its dependence on other parameters, we present following set of conditions which will be used throughout the construction below. \\

\noindent
{\bf{Choice of $r_1$:}}{\emph{
There exists a sufficiently small $r_1\leq r_0$, depending only on $k,\epsilon$, such that
\begin{equation}\label{eq-pipe:r1_cond_1}
    \frac12r_1\leq\arctn_{k-\epsilon}(\tn_k(r_1))\leq 2r_1,\ \ 
        r_1\left(7\log\frac1{r_1}+4\right)<r_0,
\end{equation}
and
\begin{equation}\label{eq-pipe:r1_cond_2}
    \frac12r\leq\sn_{k'}(r)\leq2r\ \text{ for all }\ k-1\leq k'\leq k+1\text{ and }\ 0<r\leq r_1,
\end{equation}
and the following inequalities hold for all $0<r\leq r_1$:
\begin{equation}\label{eq-pipe:r1_cond_3}
    0\leq(k-\epsilon)r^2+\frac2{\log\frac{1}{r}}+\frac1{(\log \frac{1}{r})^2}\leq1,\ \ 
    0\leq(k-\epsilon)r^2+\frac1{\log \frac{1}{r}}\leq1,
\end{equation}
\begin{equation}\label{eq-pipe:r1_cond_4}
    \sqrt{r}\log\frac1r<10^{-3},\ \ 
    \frac1{8\sqrt{r}(\log \frac{1}{r})^5}\geq 100\max\{|k|,1\}.
\end{equation}
}}

The existence of such $r_1$ follows from the asymptotics of the related functions as $r\to0$. Specifically, we have $\arctn_{k-\epsilon}(\tn_k(r))=r+O(r^2)$, $r(7\log 1/r+4)=o(1)$, $\sn_{k'}(r)=r+O(r^2)$, and the expressions in \eqref{eq-pipe:r1_cond_3} have the order $\frac{1+o(1)}{\log(1/r)}$. Finally, we may arrange for \eqref{eq-pipe:r1_cond_4} due to the fact that $\lim_{r\to0} r^a\left(\log\frac{1}{r}\right)^b=0$ for any $a,b>0$.

With our choice of $r_1$ fixed, we may proceed with the remaining pieces. \\

\noindent\textbf{Piece (ii)} is constructed as
\begin{equation}\label{eq-pipe:def_g2}
    \G_2=dr^2+\alpha^2\sn_{k-\epsilon}^2(r)d\th^2+dt^2,\quad r_3\leq r\leq r_2,\ t\in\RR.
\end{equation}
Evidently $\G_2$ has constant scalar curvature 
\begin{equation}\label{eq-pipe:scalar2}
    R_{\G_2}=2(k-\epsilon)
\end{equation}
and the warping factors are $f_2(r)=\alpha\sn_{k-\epsilon}(r)$, $u_2\equiv0$. The parameters $\alpha,r_2$ are determined by the $C^{1,1}$ matching\footnote{The system \eqref{eq-pipe:matching_12} is only a $C^1$ matching condition. It automatically follows that the concatenation of $f_1$ and $f_2$ is $C^{1,1}$ across the junction, since the continuous gluing of two Lipschitz functions of one variable is again Lipschitz. This logic also holds for the matching conditions \eqref{eq-pipe:matching_23} and \eqref{eq-pipe:matching_34} below.} of $f$ at the junction with $\G_1$:
\begin{equation}\label{eq-pipe:matching_12}
    \left\{\begin{aligned}
    & f_2(r_2)=f_1(r_1) \\
    & f'_2(r_2)=f'_1(r_1)
    \end{aligned}\right. \ \Leftrightarrow\ 
    \left\{\begin{aligned}
	& \alpha\sn_{k-\epsilon}(r_2)=\sn_k(r_1) \\
	& \alpha\cn_{k-\epsilon}(r_2)=\cn_k(r_1)
    \end{aligned}\right. \ \Rightarrow\ 
    \left\{\begin{aligned}
	& r_2=\arctn_{k-\epsilon}(\tn_k(r_1)) \\
	& \alpha=\sqrt{1-\epsilon\sn_k^2(r_1)}.
    \end{aligned}\right.
\end{equation}
The inner radius $r_3$ is determined later in Lemma \ref{lemma-pipe:c1_estimate}, where it is also made clear that $r_3<r_2$. By \eqref{eq-pipe:r1_cond_1} we have $\frac12r_1\leq r_2\leq2r_1$. By \eqref{eq-pipe:r1_cond_2} and the smallness of $\epsilon$ and $r_0$, we have $\alpha>1/2$. \\

\noindent\textbf{Piece (iii)} is the core piece in the construction. The choice of $f$ and $u$ in this piece is essentially \eqref{eq-intro:f_and_u} as mentioned in the introduction, where the function $\phi$ there is related to $u$ by $\phi=e^u$. Since piece (iii) must be glued to piece (ii) and piece (iv), the function $u$ needs to be modified near the two endpoints of the interval. The additional technical work arising in this process will occupy a large part of the construction.

Set $r_4=\varepsilon r_1^2$. Given positive constants $r_5<r_4, c_1,c_2$ to be determined below, we set
\begin{equation}\label{eq-pipe:choice_of_f_u}
    f_3(r)=r\left(1-\frac{c_1}{\log\frac{1}{r}}\right),\quad
    v_3(r)=-c_2 h(\frac r{r_4})\log\log\frac1r,\quad 
    r_5\leq r\leq r_4,
\end{equation}
where $h(x)$ is a fixed smooth cutoff function satisfying $h|_{[0,1/2]}\equiv1$, $h|_{[3/4,1]}\equiv0$ and $|h'|\leq8$. The introduction of the function $h$ is to smoothly concatenate with $u_2$ at the junction with $\G_2$.

As a preliminary step towards our goal of defining $\G_3$, consider the metric 
\begin{equation}
    \G'_3=e^{-2v_3(r)}(dr^2+f_3(r)^2d\th^2)+e^{2v_3(r)}dt^2.
\end{equation} 
To ultimately obtain the metric $\G_3$ in piece (iii), we will further modify $v_3$ near $r_5$ to facilitate gluing with the last piece. Before we explicitly describe this modification, we will make appropriate choices of parameters so that two properties are ensured: first the $C^{1,1}$ matching of $f_3$ at the junction with piece (ii), and then a large lower bound for the scalar curvature of $\G'_3$. For convenience we denote
\begin{equation}
    w(r)=\log\frac1r,\quad w_4=\log\frac1{r_4}
\end{equation}
and note the following useful computations:
\begin{equation}\label{eq-pipe:derivatives}
    f'_3=1-\frac{c_1}w-\frac{c_1}{w^2},\quad
    f''_3=-\frac{c_1(w+2)}{rw^3},\quad
    v'_3=\frac{c_2h(r/r_4)}{rw}-\frac{c_2}{r_4}h'(r/r_4)\log w.
\end{equation}

\vspace{6pt}

\noindent\textbf{The $C^{1,1}$ matching of $f$ and choice of $c_1,r_3$}. We choose the constants $c_1,r_3$ to satisfy the $C^{1,1}$ matching condition of $f$ at the junction with $\G_2$:
\begin{equation}\label{eq-pipe:matching_23}
    \begin{cases}
         f_3(r_4)=f_2(r_3) \\
         f'_3(r_4)=f'_2(r_3)
    \end{cases}
    \ \Leftrightarrow\quad
    \begin{cases}
        r_4(1-\frac{c_1}{w_4})=\alpha\sn_{k-\epsilon}(r_3) \\
	1-\frac{c_1}{w_4}-\frac{c_1}{w_4^2}=\alpha\cn_{k-\epsilon}(r_3).
    \end{cases}
\end{equation}

\begin{lemma}\label{lemma-pipe:c1_estimate}
    The system \eqref{eq-pipe:matching_23} has a solution with $\frac1{32}w_4\epsilon r_1^2\leq c_1\leq\sqrt{8r_4}w_4$ and $0<r_3\leq4\epsilon r_1^2$. In particular, $r_3<r_2$.
\end{lemma}
\begin{proof}[Proof of Lemma \ref{lemma-pipe:c1_estimate}]
To eliminate $r_3$ in the system, we square both equations and apply the Pythagorean theorem to their sum. This yields a quadratic equation $ac_1^2-2bc_1+d=0$ for the parameter $c_1$, where the constants $a,b,d$ are given by
\begin{equation}
    a=\frac1{w_4^2}+(k-\epsilon)\frac{r_4^2}{w_4^2}+\frac2{w_4^3}+\frac1{w_4^4},\ \ \ 
    b=\frac1{w_4}+\frac1{w_4^2}+(k-\epsilon)\frac{r_4^2}{w_4},\ \ \ 
    d=1-\alpha^2+(k-\epsilon)r_4^2.
\end{equation}

In order to verify the existence and claimed estimate of $c_1$, we proceed by estimating the coefficients of the quadratic equation above. Using the smallness condition \eqref{eq-pipe:r1_cond_3} of $r_1$, the following bounds hold 
\begin{equation}\label{eq-pipe:abest}
    \frac1{w_4^2}\leq a\leq\frac2{w_4^2},\quad \frac1{w_4}\leq b\leq\frac2{w_4}.
\end{equation} 
By the second equation of \eqref{eq-pipe:matching_12}, we have $d=\epsilon\sn_k^2(r_1)+(k-\epsilon)r_4^2$. Note that smallness condition \eqref{eq-pipe:r1_cond_2} implies $\frac14r_4\leq\epsilon\sn_k^2(r_1)\leq 4r_4$. Also, $|k-\epsilon|r_4^2\leq|k-\epsilon|r_0r_4\leq\frac18r_4$ by \eqref{eq-pipe:assumptions}. These three observations imply 
\begin{equation}\label{eq-pipe:dest}
    \frac18r_4\leq d\leq 8r_4.
\end{equation}

With the estimates of $a,b,d$ in hand, we are ready to find and estimate $c_1$. In what follows, we make frequent use of \eqref{eq-pipe:abest} and \eqref{eq-pipe:dest}. First note that $b^2-ad>0$, hence the above quadratic equation for $c_1$ is solvable. Using the general inequality $\sqrt{x-y}\leq\sqrt x-\frac{y}{2\sqrt x}$ for numbers $0<y<x$, we have
\begin{equation}\label{eq-pipe:c1_upper_bound}
    c_1 = \frac1a\left(b-\sqrt{b^2-ad}\right)
    \geq \frac d{2b}
    \geq \frac1{32}w_4r_4.
\end{equation}
To obtain the upper bound for $c_1$, we use the general inequality $\sqrt{x-y}\geq\sqrt x-\sqrt y$ for numbers $0<y<x$ to obtain
\begin{equation}
    c_1 \leq \sqrt{\frac da}
    \leq\sqrt{8r_4}w_4.
\end{equation}
The upper bound on $c_1$ implies $r_3>0$ from the first equation of \eqref{eq-pipe:matching_23}. Finally, we have 
\begin{equation}
    \epsilon r_1^2=r_4
    >r_4\big(1-\frac{c_1}{w_4}\big)
    =\alpha\sn_{k-\epsilon}(r_3)>\frac12\cdot\frac12r_3=\frac14r_3,
\end{equation} 
where we have used the fact that $\alpha>1/2$ and the smallness condition \eqref{eq-pipe:r1_cond_2}. The upper bound on $r_3$ follows. Since $r_2\geq\frac12 r_1$, it follows that $r_3<r_2$ and the lemma follows.
\end{proof}

\noindent\textbf{Estimating $R_{\G_3'}$ and choosing $c_2$}. Next, we make proper choice of $c_2$ to ensure that the scalar curvature of $\G_3'$ is sufficiently large. Using Proposition \ref{prop:curvatures} and \eqref{eq-pipe:derivatives} we obtain
\begin{align}
    R_{\G'_3} &= 2e^{2v_3}\left(-\frac{f''_3}{f_3}-(v'_3)^2\right) \\
    &= \frac2{w^{2c_2h(r/r_4)}}\left[\frac{c_1(w+2)}{w^2r^2(w-c_1)}-c_2^2\left(\frac{h(r/r_4)}{wr}-\frac1{r_4}h'(r/r_4)\log w\right)^2\right].\label{eq-pipe:scalarg3}
\end{align}
Since $|h'|\leq 8$ and $h'(r/r_4)=0$ when $r\leq r_4/2$, we have $h'(r/r_4)\log w(r)\leq8\log\log\tfrac{2}{\varepsilon r_1^2}$ for all $r$. Using this fact with the general inequality $(x-y)^2\leq 2x^2+2y^2$ on the square term in \eqref{eq-pipe:scalarg3}, we find
\begin{align}
    R_{\G'_3}&\geq \frac2{w^{2c_2h(r/r_4)}}\left[\frac{c_1(w+2)}{w^2r^2(w-c_1)}-2c_2^2 \frac{h(r/r_4)^2}{w^2 r^2}-2c_2^2\frac{1}{r_4^2}h'(r/r_4)^2(\log w)^2\right]\\
    &\geq \frac2{w^{2c_2h(r/r_4)}}\left[\frac{1}{w^2r^2}\left(\frac{c_1(w+2)}{(w-c_1)}-2c_2^2\right)-128c_2^2\frac{1}{r_4^2}\left(\log\log\frac 2{\epsilon r_1^2}\right)^2\right]\\
    &\geq \frac2{r^2w^{2+2c_2h(r/r_4)}}\left[\frac{c_1(c_1+2)}{w-c_1}+c_1-2c_2^2\right]-c_2^2\frac{256}{\epsilon^2r_1^4}\left(\log\log\frac 2{\epsilon r_1^2}\right)^2\label{eq-pipe:g3_scal}
\end{align}
where we have used $0\leq h\leq1$.

We first require $c_2\leq\sqrt{c_1/2}$, which is less than 1 by the first condition in \eqref{eq-pipe:r1_cond_4}. Then we can estimate the first term of \eqref{eq-pipe:g3_scal} for all $r\leq r_4$ using the bounds on $c_1$ from Lemma \ref{lemma-pipe:c1_estimate} with the fact that $0\leq h\leq1$ to find
\begin{align}
    \frac2{r^2w^{2+2c_2h(r/r_4)}}\left[\frac{c_1(c_1+2)}{w-c_1}+c_1-2c_2^2\right] &\geq
    \frac2{r^2w^{2+2c_2}}\frac{c_1(c_1+2)}{w-c_1} \label{eq-pipe:term1}\\
    &\geq \frac{4c_1}{r^2w^5}\\
    &\geq \frac18\frac{w_4\epsilon r_1^2}{r^2w^5}\\
    &\geq \frac{1}{8r^{1/2}w^5}\frac{\epsilon r_1^2}{r_4^{3/2}} \\
    &\geq \frac{1}{8r^{1/2}w^5}
\end{align}
where, in the fourth inequality we used $w_4\geq1$ with $r\leq r_4$, and in the last inequality we used $\frac{\epsilon r_1^2}{r_4^{3/2}}=\frac1{\epsilon^{1/2}r_1}\geq1$. It follows from the smallness condition \eqref{eq-pipe:r1_cond_4} that 
\begin{equation}\label{eq-pipe:scalg'part1}
    \frac2{r^2w^{2+2c_2h(r/r_4)}}\left[\frac{c_1(c_1+2)}{w-c_1}+c_1-2c_2^2\right]\geq 100\max\{|k|,1\}.
\end{equation}
In order to control the second term of the scalar curvature estimate \eqref{eq-pipe:g3_scal}, we fix a choice of $c_2\leq\min\{\sqrt{c_1/2},1/100\}$ so that
\begin{equation}\label{eq-pipe:scalg'part2}
    c_2^2\frac{256}{\epsilon^2r_1^4}\left(\log\log\frac 2{\epsilon r_1^2}\right)^2\leq\max\{{|k|},1\}.
\end{equation}
Combining \eqref{eq-pipe:scalg'part1} and \eqref{eq-pipe:scalg'part2} with \eqref{eq-pipe:g3_scal}, we conclude that 
\begin{equation}\label{eq-pipe:scalg'part3}
    R_{\G'_3}\geq99\max(|k|,1).
\end{equation}

\vspace{6pt}

\noindent\textbf{Choosing $r_5$ and modifying $v_3$ near $r_5$:} We choose the inner radius $r_5$ in \eqref{eq-pipe:choice_of_f_u} such that 
\begin{equation}\label{eq-pipe:r5_condition}
    r_5\leq\frac1{100}r_4,\quad
    v_3(2r_5)\leq\log\delta-1,\quad
    \frac{c_1}{r_5^2\left(\log\frac{1}{r_5}\right)^{1+2c_2}}\geq100\max\{|k|,1\},\quad
    \frac{f_3(r_5)}{f'_3(r_5)}\leq2r_5
\end{equation}
hold, where $\delta$ is the given parameter in condition (IV) and the last condition is achievable by the fact that $f_3(r)/f'_3(r)$ is asymptotic to $r$ at zero. These conditions will be used below to ensure desirable properties of piece (iv) and to estimate the distance in condition (VI). To finish the construction of piece (iii), we will modify $v_3$ near $r_5$ to facilitate gluing to piece (iv) defined below.





The modification is accomplished by the following lemma, whose proof is postponed to Subsection \ref{subsec:smoothing}. The main content of the lemma is that $v_3$ may be perturbed to a new function $u_3$ which is constant near $r_5$, without significantly damaging the scalar curvature lower bound. Recall that Proposition \ref{prop:curvatures} relates the left side of \eqref{eq-pipe:smoothing_u_aux1} below to the scalar curvature of our warped product metrics.

\begin{lemma}\label{lemma-pipe:smoothing_u}
    Suppose $f>0$ and $v$ are smooth functions on an interval $[s,t]$. Assume that
    \begin{equation}\label{eq-pipe:smoothing_u_aux1}
        e^{2v}\left(-\frac{f''}f-(v')^2\right)\geq\lambda
    \end{equation}
    holds on $[s,t]$ for some $\lambda\geq1$. For any $0<\mu\leq\min\big\{\frac14(t-s),1\big\}$ there is another smooth function $u$ on $[s,t]$ such that
    \begin{enumerate}
        \item $u=v$ on $[s+\mu,t]$ and $u$ is constant in $[s,s+\mu/2]$,
        \item $|u-v(s+\mu)|\leq\mu$ on $[s,s+\mu]$,
        \item $e^{2u}\left(-\frac{f''}f-(u')^2\right)\geq\lambda(1-\sqrt\mu)$ holds on $[s,t]$.
    \end{enumerate}
\end{lemma}
Apply the lemma with $f=f_3$, $v=v_3$, $s=r_5$, $t=r_4$, $\lambda=49\max\{|k|,1\}$, and $\mu=r_5$. Let $u_3$ be the function obtained, which we set as the warping factor for the third piece. We finally define
\begin{equation}\label{eq-pipe:def_g3}
    \G_3=e^{-2u_3(r)}(dr^2+f_3(r)^2d\th^2)+e^{2u_3(r)}dt^2,\quad r_5\leq r\leq r_4,
\end{equation}
as the metric for the third piece. By \eqref{eq-pipe:scalg'part3} and items 2, 3 of Lemma \ref{lemma-pipe:smoothing_u}, we have
\begin{equation}\label{eq-pipe:scalg3final}
    R_{\G_3}\geq 90\max(|k|,1),\qquad
    u_3(r_5)\leq\log\delta.
\end{equation}

\vspace{12pt}

\noindent\textbf{Piece (iv)} is defined as
\begin{equation}\label{eq-pipe:def_g4}
    \G_4=e^{-2u_3(r_5)}\big(dr^2+\sn_A^2(r)d\th^2\big)+e^{2u_3(r_5)}dt^2,\quad 0\leq r\leq r_6,\ A>0,
\end{equation}
and so $f_4(r)=\sn_A(r)$, $u_4(r)\equiv u_3(r_5)$. The parameters $A,r_6$ are determined by the $C^{1,1}$ matching of $f$:
\begin{equation}\label{eq-pipe:matching_34}
    \begin{cases}
	\sn_A(r_6)=f_3(r_5) \\
	\cn_A(r_6)=f'_3(r_5)
    \end{cases} \Rightarrow\ 
    \begin{cases}
	A=\frac{1-f'_3(r_5)^2}{f_3(r_5)^2} \\
	r_6=\arctn_A\left(\frac{f_3(r_5)}{f'_3(r_5)}\right),
    \end{cases}
\end{equation}
where we have used the Pythagorean Theorem.

We would like to estimate $A$ as expressed in \eqref{eq-pipe:matching_34}. First note that Lemma \ref{lemma-pipe:c1_estimate} and \eqref{eq-pipe:r1_cond_4} implies $c_1<\frac1{100}$. Combining this observation with the computation \eqref{eq-pipe:derivatives} and the fact that $f_3(r)\leq r$, we have
\begin{align}
    A & =f_3(r_5)^{-2}\left(-c_1^2\left[\frac1{w(r_5)^2}+\frac1{w(r_5)^4}+\frac2{w(r_5)^3}\right]+\frac{2c_1}{w(r_5)}+\frac{2c_1}{w(r_5)^2}\right)\\
    {}&\geq f_3(r_5)^{-2}\left(-\frac{c_1}{100}\cdot\frac{4}{w(r_5)^2}+\frac{2c_1}{w(r_5)}+\frac{2c_1}{w(r_5)^2}\right)\\
    &\geq f_3(r_5)^{-2}\frac{c_1}{w(r_5)} \\
    {}&\geq \frac{c_1}{r_5^2\log\frac{1}{r_5}}.\label{eq-pipe:lower_bound_A}
\end{align}
Therefore, the scalar curvature of $\G_4$ satisfies
\begin{align}
    R_{\G_4} &= 2e^{2u_3(r_5)}A\\
    &\geq 2e^{2v_3(2r_5)-2\mu}A \\
    &\geq 2e^{-2\mu}\cdot\left(\log\frac1{2r_5}\right)^{-2c_2}\cdot\frac{c_1}{r_5^2\log\frac{1}{r_5}} \\
    &\geq \frac{c_1}{r_5^2\left(\log \frac{1}{r_5}\right)^{1+2c_2}}\geq 100\max\{|k|,1\},\label{eq-pipe:scalar4}
\end{align}
where we have used Lemma \ref{lemma-pipe:smoothing_u} in the first inequality, the definition of $v_3$ in the second inequality, and the third item of \eqref{eq-pipe:r5_condition} in the final inequality.
Lastly, we will collect a necessary upper bound for $r_6$. Using the second equation of \eqref{eq-pipe:matching_34} with our choice \eqref{eq-pipe:r5_condition}, we have
\begin{equation}\label{eq-pipe:r6_r5}
    r_6 \leq \frac{f_3(r_5)}{f'_3(r_5)}
    \leq 2r_5.
\end{equation}

\subsection{Proof A step 2: verification of conditions and smoothing}\label{subsec:part2}
\vspace{6pt}

In Step 1 we constructed four pairs of functions $\{f_i,u_i\}_{i=1}^4$ in \eqref{eq-pipe:def_g1}, \eqref{eq-pipe:def_g2}, \eqref{eq-pipe:def_g3}, and \eqref{eq-pipe:def_g4}, which are defined on the intervals $[r_1,2r_1]$, $[r_3,r_2]$, $[r_5,r_4]$, and $[0,r_6]$, respectively. By translations in the $r$-coordinates, they are concatenated to form two functions $u$ and $\bar f$ defined on the interval $[0,\rho]$, where $\rho=r_1+r_2-r_3+r_4-r_5+r_6$. Due to the definition of $v_3$ and our use of Lemma \ref{lemma-pipe:smoothing_u}, the function $u$ is already smooth. According to \eqref{eq-pipe:matching_12}, \eqref{eq-pipe:matching_23}, and \eqref{eq-pipe:matching_34}, $\bar f$ is $C^{1,1}$ on $[0,\rho]$, yet may fail to be smooth. Consider the $C^{1,1}$ metric
\begin{equation}\label{eq-pipe:c11metric}
    \G'=e^{-2u}(dr^2+\bar f^2d\theta^2)+e^{2u}dt^2.
\end{equation}
To proceed, we will verify conditions (I)\,--\,(VII) for $\G'$, and afterward show this metric may be smoothed without substantive damage to the required conditions.

\vspace{6pt}

\noindent\textbf{Claim}. {\emph{The metric \eqref{eq-pipe:c11metric} defined in the above paragraph satisfies conditions (II)\,--\,(VII) of Theorem \ref{t:T3tube}, as well as the scalar curvature bound (I) wherever $\G'$ is smooth.}}

\begin{proof}[Proof of the claim]
The scalar curvature lower bound (I) was already established during each step of the above construction, namely in \eqref{eq-pipe:scalar2}, \eqref{eq-pipe:scalg'part3}, and \eqref{eq-pipe:scalar4}. Conditions (II) and (III) are evident from the form of piece (i), and condition (IV) was confirmed at the end of piece (iii)'s definition. By \eqref{e:meang}, the mean curvatures of the constant $r$ surfaces with respect to the positive $r$ direction are $e^{u(r)}\frac{\bar f'(r)}{\bar f(r)}$. Since $f'_i>0$ for all $i$ from their definitions, this verifies condition (V).
    
As for (VI), the $\G'$-distance from the central axis $\{r=0\}$ to $\p D^2\times\RR$ is equal to the sum of the integrals of $e^{-u_i}$. Since $u_i$ is constant on pieces (i), (ii), and (iv), three of these integrals are trivial. We break up the integral over piece (iii) into $[r_5,2r_5]$ and $[2r_5,r_4]$, and use the properties of $u_3$ from Lemma \ref{lemma-pipe:smoothing_u} with \eqref{eq-pipe:r6_r5} to find
\begin{align}
    d_{\G'}(\{r=0\}, \p D^2\times\RR) &= e^{-u_3(r_5)}r_6+\int_{r_5}^{r_4}e^{-u_3(r)}\,dr+(r_2-r_3)+r_1 \\
    &\leq \left(e^{-v_3(2r_5)+r_5}2r_5+\int_{r_5}^{2r_5}e^{-v_3(2r_5)+r_5}\,dr\right)+\int_{2r_5}^{r_4}e^{-v_3(r)}\,dr+3r_1\\
    &=3r_5e^{-v_3(2r_5)+r_5}+\int_{2r_5}^{r_4}e^{-v_3(r)}\,dr+3r_1.\label{eq-pipe:distanceest}
\end{align}
Since both $c_2<1$ and $h<1$, we have $e^{-v_3(r)}\leq\log\frac1r$ for all $r<1/e$. Also note that \eqref{eq-pipe:r5_condition} implies $e^{r_5}<2$. We use these facts and an explicit integral calculation to continue the estimate on \eqref{eq-pipe:distanceest}
\begin{align}
    d_{\G'}(\{r=0\}, \p D^2\times\RR) &\leq 3r_5e^{r_5}e^{-v_3(2r_5)}+\int_{0}^{r_1}\log\frac1r\,dr+3r_1 \\
    &\leq 6r_5\log\frac1{r_5}+\int_0^{r_1}\log\frac1r\,dr+3r_1 \\
    &\leq 6r_1\log\frac1{r_1}+r_1\left(1+\log\frac1{r_1}\right)+3r_1, \label{eq-pipe:distance}
\end{align}
where in the last line we used the fact that $r\mapsto r\log\frac1r$ is increasing for $r\in (0,1/e)$ on the first term. According to the smallness condition \eqref{eq-pipe:r1_cond_1}, the above estimate implies $d_{\G'}(\{r=0\}, \p D^2\times\RR)<r_0$, as desired.

    
For condition (VII), the volume of $D^2\times[0,1]$ is equal to the total integral of $e^{-u_i}f_i$ along the four pieces:
\begin{equation}\label{e:volume1}\begin{aligned}
    \frac1{2\pi}\vol(D^2\times[0,1],\G') &= e^{-u_3(r_5)}\int_0^{r_6}\sn_A(r)\,dr+\int_{r_5}^{r_4}e^{-u_3(r)}f_3(r)\,dr \\
    &\qquad +\int_{r_3}^{r_2}\alpha\sn_{k-\epsilon}(r)\,dr + \int_{r_1}^{2r_1}\sn_k(r)\,dr.
\end{aligned}\end{equation}
We use the smallness condition \eqref{eq-pipe:r1_cond_2} and $r_2\leq 2r_1$ to estimate the last two terms in \eqref{e:volume1}:
\begin{align}\label{e:last2est}
\begin{split}
    \int_{r_3}^{r_2}\alpha\sn_{k-\epsilon}(r)\,dr + \int_{r_1}^{2r_1}\sn_k(r)\,dr&\leq r_2^2+4r_1^2 \leq 8r_1^2.
\end{split}
\end{align}
For the second term in \eqref{e:volume1}, we use the properties of $u_3$ from Lemma \ref{lemma-pipe:smoothing_u} and $e^{r_5}<2$ with the basic estimates $f_3(r)\leq r$ and $e^{-v_3(r)}\leq \log\frac1r$ to find
\begin{align}\label{eq-pipe:secondterm}
        \int_{r_5}^{r_4}e^{-u_3(r)}f_3(r)\,dr &\leq \int_{r_5}^{2r_5}e^{-v_3(2r_5)+r_5}f_3(r)\,dr + \int_{2r_5}^{r_4}e^{-v_3(r)}f_3(r)\,dr\\
        {}&\leq 4r_5^2\log\frac1{2r_5} +\int_{2r_5}^{r_4}r\log\frac1r\,dr\\
        {}&\leq 4r_4^2\log\frac1{r_4}+\frac14 r_4^2\left(2\log\frac{1}{r_4} +1\right)
\end{align}
where we have used the fact that $r\mapsto r^2\log\tfrac1r$ is increasing for $r<e^{-1/2}$ in the last line. The first term in \eqref{e:volume1} can be estimated using the fact that $\sn_A(r)\leq r$ with the bounds \eqref{eq-pipe:r5_condition} and \eqref{eq-pipe:r6_r5}:
\begin{align}\label{eq-pipe:firstterm}
    e^{-u_3(r_5)}\int_0^{r_6}\sn_A(r)\,dr\leq 2\left(\log\frac1{2r_5}\right)\frac12r_6^2 \leq r_4r_5\log\frac1{r_5}\leq r_4^2\log\frac1{r_4}.
\end{align}
Combining \eqref{e:volume1}, \eqref{e:last2est}, \eqref{eq-pipe:secondterm}, and \eqref{eq-pipe:firstterm} with $r_4=\epsilon r_1^2$ yields
\begin{align}
    \vol(D^2\times[0,1],\G')&\leq 2\pi \left(6r_4^2\log\frac1{r_4}+10r_1^2\right)\\
    &\leq 2\pi \left(6\varepsilon r_4\log\frac1{r_4}+10\right)r_1^2\\
    &\leq 2\pi(6\epsilon+10)r_1^2
\end{align}
where we have used the fact that $r\log\tfrac1{r_4}<1$.
On the other hand, we have $V_0=2\pi\int_0^{2r_1}\sn_k(r)\,dr\geq \frac12\pi r_1^2$ by \eqref{eq-pipe:r1_cond_2}, and therefore condition (VII) is satisfied. This finishes the proof of the claim.
\end{proof}

The only remaining task is to improve the regularity of $\G'$ from $C^{1,1}$ to $C^\infty$ without significantly damaging its main properties. More precisely, we make a small perturbation of $\bar f$ to smooth out the discontinuities in $\bar f''$ across the junctions between neighboring pieces, and then form the desired smooth metric \eqref{eq-pipe:general_form}. This is accomplished by making use of the following lemma, whose proof is left to Subsection \ref{subsec:smoothing}.
\begin{lemma}\label{lemma-pipe:smoothing_f}
    Suppose we are given a function $u\in C^\infty((-t,t))$ and a positive function $\bar f\in C^{1,1}((-t,t))$ which is smooth on both $(-t,0]$ and $[0,t)$. Assume that 
    \begin{equation}
        e^{2u}\left(-\frac{\bar f''}{\bar f}-(u')^2\right)\geq\lambda
    \end{equation} 
    holds on both $(-t,0)$ and $(0,t)$ for some $\lambda\in\mathbb{R}$.\footnote{Since $\bar f$ is smooth up to the left and right side of $0$, this inequality also holds at $0$ in the sense of one-sided derivatives.}
    Given any small $\mu\in(0,t/2)$, there exists another positive function $f\in C^\infty((-t,t))$ such that
\begin{enumerate}
    \item $||f-\bar f||_{C^1((-t,t))}\leq\mu$, and $f=\bar f$ outside $(-\mu,\mu)$,
    \item $e^{2u}\left(-\frac{f''}{f}-(u')^2\right)\geq\lambda-\mu$ holds on $(-t,t)$.
    \end{enumerate}
\end{lemma}
Apply Lemma \ref{lemma-pipe:smoothing_f} at each of the three junction points of $\bar f$, with $\lambda=k-\epsilon$ and
\begin{equation}
    \mu<10^{-3}\min\{\epsilon,r_6,r_5-r_4,r_3-r_2,r_1\}
\end{equation}
to be chosen sufficiently small below. Let $f$ be the function obtained, and define the smooth metric $\G=e^{-2u}(dr^2+f^2d\theta^2)+e^{2u}dt^2$ for $r\in[0,\rho]$. By item 2 in Lemma \ref{lemma-pipe:smoothing_f}, we have $R_{\G}\geq 2(k-2\epsilon)$. We have already shown that $\G'$ satisfies conditions (II)\,--\,(VII) of Theorem \ref{t:T3tube}, and it remains to verify that $\G$ satisfies these conditions as well. Condition (II) is not affected by the perturbation of $f$ since $\mu<\frac12r_1$. Conditions (III) and (IV) do not involve $f$ thus continue to hold. Conditions (V)\,--\,(VII) are open and depend $C^1$-continuously on $f$, and therefore continue to hold for $f$ if $\mu$ is sufficiently small. This completes the proof of Theorem \ref{t:T3tube} with $2\epsilon$ in place of $\epsilon$, from which the original statement quickly follows.
\end{proof}

\subsection{Proof of the smoothing lemmas}\label{subsec:smoothing}
Let us prove the lemma used in the construction of $u_3$ in piece (iii) above.

\begin{proof}[Proof of Lemma \ref{lemma-pipe:smoothing_u}]
    By further decreasing $\mu$ and using the boundedness of $v'$, we may assume that
    \begin{equation}
        \sup_{s\leq r\leq s+\mu}|v(r)-v(s+\mu)|\leq\frac14\sqrt\mu.\label{eq-pipe:smoothing_u_aux0}
    \end{equation}
    Fix a smooth function $\eta(r)$ on $[s,t]$ such that $\eta|_{[s,s+\mu/2]}\equiv0$, $\eta|_{[s+\mu,t]}\equiv1$, and $0\leq\eta(r)<1$ when $r\in[s+\mu/2,s+\mu)$. Leveraging this last property, by the Monotone Convergence Theorem there exists sufficiently large $q\in\RR$ such that
    \begin{equation}\label{eq-pipe:smoothing_aux1}
        \int_s^{s+\mu}\eta(\rho)^q\,d\rho\leq\frac{\mu}{4\sup_{[s,t]}|v'|}.
    \end{equation}
    For this choice of $q$, define the smooth function
    \begin{equation}
        u(r)=v(s+\mu)+\int_{s+\mu}^r\eta(\rho)^qv'(\rho)\,d\rho,\qquad s\leq r\leq t.
    \end{equation}
    According to \eqref{eq-pipe:smoothing_aux1}, we have 
    \begin{equation}\label{eq-pipe:smoothing_u_aux3}
        \sup_{r\in[s,s+\mu]}|u(r)-v(s+\mu)|\leq\frac14\mu,
    \end{equation} which verifies condition 2 in the lemma. From the construction, we have $u=v$ on $[s+\mu,t]$, and $u$ is constant on $[s,s+\mu/2]$, which verifies condition 1. 
    
    It remains to show condition 3. For $s+\mu\leq r\leq t$, where $u(r)=v(r)$, the desired lower bound follows directly from the hypothesis \eqref{eq-pipe:smoothing_u_aux1}. For $s\leq r\leq s+\mu$, using $|u'|=|\eta^qv'|\leq|v'|$ and \eqref{eq-pipe:smoothing_u_aux1} we have
    \begin{equation}
        e^{2u}\Big(-\frac{f''}f-(u')^2\Big)
        \geq e^{2u}\Big(-\frac{f''}f-(v')^2\Big)
        \geq e^{2u}\cdot\lambda e^{-2v}
        \geq\lambda e^{-2|u-v|}.\label{eq-pipe:smoothing_u_aux2}
    \end{equation}
    Notice that, for any $r\in[s,s+\mu]$, inequalities \eqref{eq-pipe:smoothing_u_aux0} and \eqref{eq-pipe:smoothing_u_aux3} imply
    \begin{align}
        |u(r)-v(r)| &\leq |u(r)-v(s+\mu)|+|v(s+\mu)-v(r)| \\
        &\leq \frac14\mu+\frac14\sqrt\mu \\
        &\leq \frac12\sqrt\mu,
    \end{align}
    where we have used the assumption $\mu\leq1$ in the last inequality.
    Since $e^{-\sqrt\mu}\geq 1-\sqrt\mu$, condition 3 follows from combining the above inequality with \eqref{eq-pipe:smoothing_u_aux2}.
\end{proof}


The final task of the section is to prove the lemma used at the end of the proof of Theorem \ref{t:T3tube} to smooth the $f_i$ functions.

\begin{proof}[Proof of Lemma \ref{lemma-pipe:smoothing_f}]

While the existence of a smooth function $\tilde{f}$ satisfying property $1$ is standard, achieving the scalar curvature bound $2$ requires explanation. We proceed by reviewing the standard mollification constructions, establishing property $1$, then proving property $2$. Let $\eta\in C^\infty(\mathbb{R})$ be a positive even function with support contained in $[-1,1]$ and such that $\int_{-1}^1\eta(r)\,dr=1$. For positive $\epsilon\ll t/8$, define the mollifier $\eta_\epsilon(r)=\epsilon^{-1}\eta(\epsilon^{-1}r)$, thus $\int_{-\epsilon}^\epsilon\eta_\epsilon(r)\,dr=1$. Consider the standard mollification $\bar f_\epsilon$ given by convolution $\bar f_\epsilon=\bar f*\eta_\epsilon$, which is defined on $(-t+\epsilon,t-\epsilon)$. Note that $\bar f'_\epsilon=\bar f'*\eta_\epsilon$. It follows that, as $\epsilon\to0$, $\bar f_\epsilon$ converges to $\bar f$ in $C^1([-t/2,t/2])$ and in $C^\infty\big([-3t/4,-t/4]\cup[t/4,3t/4]\big)$. Since $\bar f'$ is Lipschitz across $0$, we have $\bar f''_\epsilon=\bar f''*\eta_\epsilon$. This implies
\begin{equation}\label{eq-pipe:smoothing_f2}
    \inf_{|y-x|\leq\epsilon, y\ne0}\bar f''(y)\leq \bar f''_\epsilon(x)\leq\sup_{|y-x|\leq\epsilon, y\ne0}\bar f''(y),\quad\text{ for } x\in[-3t/4,3t/4].
\end{equation}

Next, in order to avoid disturbing $\bar f$ away from a neighborhood of $0$, we introduce a smooth cutoff function $h$ satisfying $h\equiv1$ on $[-\mu/2,\mu/2]$, $h\equiv0$ on $[-t,-\mu]\cup[\mu,t]$, and $0\leq h\leq1$. Consider $f_\epsilon=h\bar f_\epsilon+(1-h)\bar f$, which well approximates $\bar f$ in $C^1((-t,t))$ and in the $C^2$ sense away from $0$. Hence
    \begin{equation}\label{eq-pipe:smoothing_f3}
    \inf_{|y-x|\leq\epsilon, y\ne0}\bar f''(y)-o_\epsilon(1)\leq f''_\epsilon(x)\leq\sup_{|y-x|\leq\epsilon, y\ne0}\bar f''(y)+o_\epsilon(1),\quad\text{ for }  x\in[-3t/4,3t/4],
    \end{equation}
where, above and throughout, $o_\epsilon(1)$ denotes a term which converges uniformly to $0$ for all $x\in[-3t/4,3t/4]$ as $\epsilon\to0$. Let $a>0$ be such that $\inf_{[-3t/4,3t/4]}\bar f\geq a$ and let $b>0$ be such that both $\sup_{[-3t/4,3t/4]}e^{2u}\leq b$ and $||\bar f||_{C^{1,1}([-3t/4,3t/4])}\leq b$. Then for sufficiently small $\epsilon$ we have $f_\epsilon\geq a/2>0$ on $[-3t/4,3t/4]$. Property 1 follows. As a consequence, we have $1/f_\epsilon\to 1/\bar f$ uniformly on $[3t/4,3t/4]$.

To show property $2$, fix a point $x\in[-t/2,t/2]$ and let $J=[x-\epsilon,x+\epsilon]\setminus\{0\}$. To begin, we use \eqref{eq-pipe:smoothing_f3} and the properties of $a,b$ to find
\begin{align}
    e^{2u(x)}\Big(-\frac{f''_\epsilon(x)}{f_\epsilon(x)}-u'(x)^2\Big)
    &\geq -\sup_{y\in J} e^{2u(x)}\frac{\bar f''(y)+o_\epsilon(1)}{f_\epsilon(x)}
        -e^{2u(x)}u'(x)^2 \\
    &\geq -e^{2u(x)}\frac{o_\varepsilon(1)}{f_\epsilon(x)}
        -\sup_{y\in J} e^{2u(x)}\frac{\bar f''(y)}{f_\epsilon(x)}
        -e^{2u(x)}u'(x)^2 \label{eq-pipe:smoothing_fscal1}\\
    &\geq -b\cdot \frac2a\cdot o_\epsilon(1)
        -\sup_{y\in J} e^{2u(x)}\bar f''(y)\left(\frac1{f_\epsilon(x)}-\frac1{\bar f(x)}\right) \nonumber\\
    &\qquad -\sup_{y\in J} \bar f''(y)\left(\frac{e^{2u(x)}}{\bar f(x)}-\frac{e^{2u(y)}}{\bar f(y)}\right)
        -\sup_{y\in J} e^{2u(y)}\left(\frac{\bar f''(y)}{\bar f(y)}+u'(y)^2\right) \nonumber\\
    &\qquad +\left(\inf_{y\in J} e^{2u(y)}u'(y)^2-e^{2u(x)}u'(x)^2\right),\label{eq-pipe:smoothing_fscal2}
\end{align}
where in the last inequality, we have added and subtracted three quantities to the second term of \eqref{eq-pipe:smoothing_fscal1} and used the facts that, for any two functions $A$ and $B$, we have $\sup(A+B)\leq\sup A+\sup B$ and $-\sup(-A)=\inf A$. Using the properties of $a,b$, the Fundamental Theorem of Calculus, and the original scalar curvature bound by $2\lambda$ we continue the above estimate to find 
\begin{align}
    \begin{split}e^{2u(x)}\Big(-\frac{ f''_\epsilon(x)}{ f_\epsilon(x)}-u'(x)^2\Big)
    &\geq -2a^{-1}b\cdot o_\epsilon(1)
        -b\cdot b\cdot ||1/ f_\epsilon-1/\bar f||_{C^0([-3t/4,3t/4])} \\
    &\qquad -b\cdot\epsilon\cdot||e^{2u}/\bar f||_{C^1([-3t/4,3t/4])}
        +\lambda \\
    &\qquad -\epsilon\cdot||e^{2u}(u')^2||_{C^1([-3t/4,3t/4])}\end{split}\\
    &\geq \lambda-o_\epsilon(1),
\end{align}
where in the last inequality, we have used the remark at the end of the previous paragraph. One may then choose $\epsilon$ sufficiently small such that condition 2 holds, proving the lemma.
\end{proof}

\section{Construction by cutoff functions}\label{sec:analytic}

In this section, we give the second proof, Proof B, of Theorem \ref{t:T3tube} by directly concatenating the main metric \eqref{eq-intro:f_and_u} with the product metric $S^2_k\times\RR$ via cutoff functions.

\begin{remark}
    The construction presented in this section actually proves a slightly stronger version of Theorem \ref{t:T3tube}. In particular, several properties in Theorem \ref{t:T3tube} can be improved due to this proof. First, when $r\in[r_1,2r_1]$ the metric is exactly equal to $dr^2+\sn_k(r)^2d\th^2+dt^2$, allowing us to take the identity as the isometry mentioned in property $(II)$. Property $(IV)$ can be improved to the statement $e^{u}\equiv\delta$ in a neighborhood of $r=0$, and item $(VI)$ can be strengthened to $d_{\G}(\{0\}\times\RR,\p D^2\times\RR)\leq 3r_1$. 
\end{remark}


The construction is as follows. We fix two smooth functions $\zeta$ and $\eta$ satisfying
\begin{align}\label{eq-cutoff:etazeta}
\begin{split}
    & \zeta|_{[0,\frac12]}=0,\quad \zeta|_{[1,\infty)}=1,\quad    0\leq\zeta'\leq4,\quad |\zeta''|\leq16,\\
    & \eta|_{[0,\frac12]}=1,\quad \eta|_{[1,\infty)}=0,\quad 0\geq\eta'\geq-4, \quad|\eta''|\leq16.
\end{split}
\end{align} 
For a given $k\in\mathbb{R}$ and positive constants $r_1,r_2,c_1,c_2$ with $r_2\ll r_1$, we set
\begin{align}
    \psi(r)&=\int_0^r\Big[\zeta\big(\frac s{r_2}\big)\frac{1}{s\log^2(1/s)}+\big(1-\zeta\big(\frac s{r_2}\big)\big)\frac s{r_2}\Big]\,ds, \label{eq-cutoff:def_psi}\\
    h(r)&=1-c_1\eta\big(\frac r{r_1}\big)\psi(r), \label{eq-cutoff:def_h}
\end{align}
and consider the functions
\begin{align}
    f(r) &= \sn_k(r)\cdot h(r)\label{eq-cutoff:def_f}\\
    u(r) &= -c_2\int_r^\infty\eta\big(\frac{4s}{r_1}\big)\zeta\big(\frac s{4r_2}\big)\frac{ds}{s\log(1/s)}.\label{eq-cutoff:def_u}
\end{align}

Here are some basic facts concerning these expressions. Since $f(r)=\sn_k(r)$ and $u(r)\equiv0$ in $[r_1,2r_1]$, the metric $\G=e^{-2u}(dr^2+f^2d\th^2)+e^{2u}dt^2$ smoothly concatenates with the product $S^2_k\times\RR$. Also, we find that $u(r)$ is constant and $f(r)=\sn_k(r)\cdot(1-c_1r^2/2r_2)$ in $[0,r_2/2]$, thus $\G$ is smooth near the origin. For the case $k=0$, by replacing the cutoff functions $\eta,\zeta$ by the constant function 1 we exactly recover the formulas \eqref{eq-intro:f_and_u}, with the identification $\varphi=e^u$. The function $\eta$ plays the role of concatenating \eqref{eq-intro:metric} with the exterior product metric, while $\zeta$ plays the role of smoothing the singularity at $r=0$. As such, Figure \ref{fig:graph1} also serves as a visualization of $\G$.


The parameters $r_1,r_2,c_1,c_2$ need to be specified. The following lemma determines the choice of $r_1$.
\begin{lemma}\label{lemma-cutoff:r1}
    Given $r_0>0$ and $k\in\mathbb{R}$, there exists $r_1>0$ such that
    \begin{enumerate}[1.]
        \item $r_1<r_0$ and $r_1<\frac1{100(1+|k|)}$ (in particular, $\log(1/r_1)>4$ and $\log\log(1/r_1)>0$),
        \item for all $0<r\leq r_1$ we have
        \begin{equation}\label{eq-cutoff:sn}
            \frac12r\leq\sn_k(r)\leq2r,\quad \Big|\frac{\cn_k(r)}{\sn_k(r)}-\frac1r\Big|\leq|k|r,\quad
        \end{equation}
        \item for all $0<r\leq r_1$, $\frac12\leq a\leq2$ and $1\leq b\leq 5$ we have
        \begin{equation}\label{eq-cutoff:rawb}
            \frac1{r^a\big(\log\tfrac{1}{r}\big)^b}>100(|k|+1).
        \end{equation}
    \end{enumerate}
\end{lemma}
\begin{proof}
    This follows from the asymptotics of the functions involved near $r=0$. Indeed, for fixed $k$ we have $\sn_k(r)=r+O(r^3)$ and $\frac{\cn_k(r)}{\sn_k(r)}=\frac1r-\frac k3r+O(r^3)$ when $r\to0$, leading to item 2. To establish item 3, first note that $\lim_{r\to0}\frac1{r^{1/2}\log^5(1/r)}=\infty$. Moreover, if $\frac12\leq a\leq2$, $1\leq b\leq 5$, and $r<1/e$, then $r^a\log^b(1/r)\leq r^{1/2}\log^5(1/r)$. Then $r_1$ may be chosen sufficiently small to accomplish all items.
\end{proof}

The following technical lemma is useful.

\begin{lemma}\label{lemma-cutoff:psi}
    Suppose $0<r_2<r_1<e^{-4}$ and consider the function $\psi$ defined by \eqref{eq-cutoff:def_psi}. Then $\psi(r)\leq\frac12$ for all $r\leq r_1$.
\end{lemma}
\begin{proof}
    We have
    \[\psi(r)\leq\int_0^r\frac{ds}{s\log^2{1/s}}+\int_0^{\min(r,r_2)}\frac s{r_2}\,ds\leq\frac1{\log(1/r)}+\frac{r_2}2,\]
    which is less than $1/2$ by our assumptions.
\end{proof}

Finally, the following lemma determines the choice of $c_1,c_2,r_2$.
\begin{lemma}\label{lemma-cutoff:c1c2r2}
Given $\epsilon,\delta>0$, $r_1\in(0,e^{-4})$, and $k\in\mathbb{R}$, there exists a choice of $c_1,c_2>0,$ and $r_2\in(0,r_1)$ such that
\begin{enumerate}[1.]
    \item the function $h$ defined by \eqref{eq-cutoff:def_h} satisfies
    \begin{equation}
        \inf_{[r_1/4,r_1]}\Big(-\frac{h''}{h}-2\frac{h'}{rh}-2|k|r\frac{h'}h\Big)\geq-\epsilon,
    \end{equation}
    \item $c_2\leq c_1<\min\{r_1,\frac1{100(1+|k|)}\}$ and $c_2\big[\log\log(\frac1{\min\{r_1/64,c_1^2\}})-\log\log(\frac1{r_1})\big]<\log(1/\delta)$,
    \item $r_2<\min\{\frac1{64}r_1,c_1^2\}$ and satisfies
    \begin{equation}\label{eq-cutoff:choice_of_r2}
        c_2\int_0^\infty\eta\big(\frac{4s}{r_1}\big)\zeta\big(\frac s{4r_2}\big)\frac{ds}{s\log(1/s)}=\log\frac1\delta.
    \end{equation}
\end{enumerate}
\end{lemma}
\begin{proof}
    We first establish some preliminary observations which are used below without mention. Lemma \ref{lemma-cutoff:psi} states $\psi\leq\frac12$ on $[0,r_1]$. By taking $c_1<1$, this implies $h\geq\frac12$. Next, we note that by taking $r_2<\frac1{64}r_1$, we have $\zeta(r/r_2)\equiv1$ on $[r_1/4,r_1]$. This implies that  $\psi'(r)=\frac1{r\log^2(1/r)}$ for $r\in[r_1/4,r_1]$. In particular, $\psi'$ and $\psi''$ are independent of $r_2$.
    
    To establish item 1 of the lemma we directly compute the derivatives of $h(r)=1-c_1\eta(r/r_1)\psi(r)$:
    \begin{align}
        &h'(r)=-c_1r_1^{-1}\eta'\big(\frac r{r_1}\big)\psi(r)-c_1\eta(r/r_1)\psi'(r),\\
        &h''(r)=-c_1r_1^{-2}\eta''\big(\frac r{r_1}\big)\psi(r)
            -2c_1r_1^{-1}\eta'\big(\frac r{r_1}\big)\psi'(r)
            -c_1\eta(r/r_1)\psi''(r).
    \end{align}
    From this and \eqref{eq-cutoff:etazeta}, we obtain the bounds
    \begin{align}
        &|h'(r)|\leq 2c_1r_1^{-1}+c_1\psi'(r)\\
        &|h''(r)|\leq 8c_1r_1^{-2}+8c_1r_1^{-1}\psi'(r)+c_1|\psi''(r)|.
    \end{align}
    This leads to the inequality
    \begin{align}
        -\frac{h''}{h}-2\frac{h'}{rh}-2|k|r\frac{h'}h &\geq -2|h''|-4r^{-1}|h'|-4|k|r|h'| \\
        &\begin{aligned}
            &\geq -2\Big[8c_1r_1^{-2}+8c_1r_1^{-1}\psi'(r)+c_1|\psi''(r)|\Big] \\
            &\qquad -\big(4r^{-1}+4|k|r\big)\Big[2c_1r_1^{-1}+c_1\psi'(r)\Big].
        \end{aligned} \label{eq-cutoff:aux1}
    \end{align}
    The right hand side of \eqref{eq-cutoff:aux1} depends only on $r,c_1,r_1$ and $k$ but not on $r_2$, as long as $r_2<\frac1{64}r_1$ and $r\in[\frac{r_1}4,r_1]$. As $r_1$ and $k$ are fixed, only $c_1,r$ enter as variables. Thus, we use $F(c_1,r)$ to denote the right hand side of \eqref{eq-cutoff:aux1}. By continuity, the function
    \begin{equation}
        G(c_1)=\min_{r\in[r_1/4,r_1]}F(c_1,r)
    \end{equation}
    is continuous for $c_1\in[0,1]$. Since $G(0)=0$, we may choose a sufficiently small $c_1$ (depending only on $r_1$ and $k$) such that $c_1<\min\{r_1,\frac1{100(|k|+1)}\}$ and $G(c_1)\geq-\epsilon$. We fix such a choice. This choice establishes item 1 of the lemma so long as $r_2<\frac1{64}r_1$ holds.

    The remaining task is to choose $c_2$ and $r_2$ which satisfy items 2 and 3 of the lemma. In particular, once this is achieved we obtain $r_2<\frac1{64}r_1$ which ensures item 1. Note that there exists a choice of $c_2>0$ (depending on $c_1, k, r_1, \delta$) such that item 2 is satisfied. We fix this choice. Denote $\bar r=\min\{\frac1{64}r_1,c_1^2\}$. Consider the continuous function
    \begin{equation}
        I(r_2)=c_2\int_0^\infty\eta\big(\frac{4s}{r_1}\big)\zeta\big(\frac s{4r_2}\big)\frac{ds}{s\log(1/s)},\quad r_2\in(0,\bar r].
    \end{equation}\updatetag{v3: fixed typo}
    By item 2 we have
    \begin{equation}
        I(\bar r)\leq c_2\int_{\bar r}^{r_1}\frac{ds}{s\log(1/s)}=c_2\Big[\log\log\frac1s\Big]_{\bar r}^{r_1}<\log(\frac1\delta).
    \end{equation}
    On the other hand, we have
    \begin{equation}
    \liminf_{r_2\to0^+}I(r_2)\geq c_2\liminf_{r_2\to0^+}\int_{4r_2}^{r_1/4}\frac{ds}{s\log(1/s)}=+\infty.
    \end{equation}
    In light of the last two inequalities, we may choose $r_2\in(0,\bar r)$ such that $I(r_2)=\log(1/\delta)$. This establishes item 3.
\end{proof}

\begin{proof}[Proof B of Theorem \ref{t:T3tube}] {\ }

Given $\varepsilon,\delta,k$ and $r_0$, we fix $r_1$ as specified by Lemma \ref{lemma-cutoff:r1} and then fix $c_1,c_2,$ and $r_2$ as specified by Lemma \ref{lemma-cutoff:c1c2r2}. Define the functions $u$ and $f$ by \eqref{eq-cutoff:def_f} and \eqref{eq-cutoff:def_u} with $r\in[0,2r_1]$. It remains to show that the metric $\G=e^{-2u}(dr^2+f^2d\th^2)+e^{2u}dt^2$ satisfies conditions $(I)$-$(VII)$ of Theorem \ref{t:T3tube}. As explained below \eqref{eq-cutoff:def_u}, $\G$ is smooth and satisfies conditions $(II)$ and $(III)$. Condition $(IV)$ follows from item 3 of Lemma \ref{lemma-cutoff:c1c2r2}; in fact we have $e^{u(0)}=\delta$ by \eqref{eq-cutoff:def_u} and \eqref{eq-cutoff:choice_of_r2}. The first task, occupying the bulk of the proof, is to verify the scalar curvature lower bound in condition $(I)$. Afterwards, the distance, volume, and mean convexity conditions (V), (VI), and (VII) are established.

\vspace{6pt}

\textbf{Preliminary bounds for $u$.} The following estimates are useful in later arguments. Note that
    \begin{equation}\label{eq-cutoff:u_bound}
        u(r)\geq -c_2\int_r^{r_1}\frac{ds}{s\log(1/s)}=-c_2\log\log\frac1r+c_2\log\log\frac1{r_1}.
    \end{equation}
Dropping the second term (which is positive by item 1 of Lemma \ref{lemma-cutoff:r1}), we obtain $e^{u(r)}\geq(\log\frac1r)^{-c_2}$. Since $u$ is constant in $[0,r_2]$, on this interval we obtain $e^{u(r)}\geq e^{u(r_2)}\geq(\log\frac1{r_2})^{-c_2}\geq(\log\frac1{r_2})^{-1}$.
    

\vspace{6pt}
    
\textbf{Scalar curvature lower bound.} According to Proposition \ref{prop:curvatures}, we need to establish the inequality
    \begin{equation}\label{eq-cutoff:scalar_curvature}
        e^{2u}\Big(-\frac{f''}f-(u')^2\Big)\geq(k-\epsilon),\qquad r\in[0,r_1].
    \end{equation}
By the fact $f(r)=\sn_k(r)\cdot h(r)$, we have
    \begin{align}
        -\frac{f''}f-(u')^2 &= k-\frac{h''}h-2\frac{\cn_k}{\sn_k}\frac{h'}h-(u')^2 \\
        &\geq k+\Big[-\frac{h''}h-2\frac{h'}{rh}-(u')^2\Big]-2|k|r\frac{|h'|}{h} ,\label{eq-cutoff:reduction_to_h}
    \end{align}
where we have used item 2 of Lemma \ref{lemma-cutoff:r1} in the last line. To prove \eqref{eq-cutoff:scalar_curvature}, we proceed by considering the intervals $[0,\frac14 r_1]$ and $[\frac14r_1,r_1]$ separately.

Assume $r\in[\frac14r_1,r_1]$. For such $r$, we have $u(r)\equiv0$. Thus \eqref{eq-cutoff:reduction_to_h} and item 1 of Lemma \ref{lemma-cutoff:c1c2r2} imply
    \begin{equation}\label{eq-cutoff:scal_interval2}
        e^{2u}\Big(-\frac{f''}f-(u')^2\Big)\geq k-\frac{h''}h-2\frac{h'}{rh}-2|k|r\frac{|h'|}{h}\geq k-\epsilon,
    \end{equation}
establishing the scalar curvature bound on this interval.
    
Now suppose $r\in[0,\frac14r_1]$. On this interval we have $\eta(r/r_1)\equiv1$, which simplifies the computation for $h$. For convenience, we set
\begin{equation}
    w:=\log(1/r)
\end{equation}
and record the following direct computations
    \begin{align}\label{eq-cutoff:dh}
        h'(r) &= -c_1\zeta\big(\frac r{r_2}\big)\frac1{rw^2}
        -c_1\big(1-\zeta\big(\frac r{r_2}\big)\big)\frac r{r_2}\\
        \begin{split}\label{eq-cutoff:ddh}
        h''(r) &= -\frac{c_1}{r_2}\zeta'\big(\frac r{r_2}\big)\frac1{rw^2}
                +c_1\zeta\big(\frac r{r_2}\big)\frac1{r^2w^2}
                -2c_1\zeta\big(\frac r{r_2}\big)\frac1{r^2w^3} \\
                &\qquad +\frac{c_1}{r_2}\zeta'\big(\frac r{r_2}\big)\frac r{r_2}
                -c_1\big(1-\zeta\big(\frac r{r_2}\big)\big)\frac 1{r_2}.
        \end{split}
    \end{align}
    Also note that
    \begin{equation}\label{eq-cutoff:du}
        |u'|=\frac{c_2}{rw}\eta\big(\frac{4r}{r_1}\big)\zeta\big(\frac r{4r_2}\big)\leq\frac{c_2}{rw}\zeta\big(\frac r{r_2}\big).
    \end{equation}
    
    In what follows, we adopt the shorthand  $\zeta=\zeta(r/r_2)$ and $\zeta'=\zeta'(r/r_2)$. Combining \eqref{eq-cutoff:dh}, \eqref{eq-cutoff:ddh}, and \eqref{eq-cutoff:du} we estimate the middle term of \eqref{eq-cutoff:reduction_to_h}
    \begin{align}
        \begin{split}-\frac{h''}h-\frac{2h'}{rh}-(u')^2 &\geq \frac1{1-c_1\psi}\Big[
            \frac{c_1}{r_2}\zeta'\frac1{rw^2}
            -c_1\zeta\frac1{r^2w^2}
            +2c_1\zeta\frac1{r^2w^3}
            -\frac{c_1}{r_2}\zeta'\frac r{r_2}
            +c_1(1-\zeta)\frac1{r_2}\Big], \\
        &\qquad +\frac{2}{r(1-c_1\psi)}\Big[c_1\zeta\frac1{rw^2}+c_1(1-\zeta)\frac r{r_2}\Big]-\frac{c_2^2}{r^2w^2}\zeta^2 
        \end{split}\\
        &= \frac{c_1}{1-c_1\psi}\Big[\zeta'\cdot(\frac1{r_2rw^2}-\frac r{r_2^2})+\frac{\zeta}{r^2w^2}+\frac{2\zeta}{r^2w^3}+3\frac{1-\zeta}{r_2}\Big]-\frac{c_2^2}{r^2w^2}\zeta^2.\label{eq-cutoff:0to18p1}
    \end{align}
    To proceed, notice that item 3 of Lemma \ref{lemma-cutoff:r1} implies $\frac1{rw^2}>1\geq\frac r{r_2}$ holds on $[0,r_2]$. Since $\zeta'\geq0$ and $\zeta$ is constant on $[r_2,r_1]$, it follows that $\zeta'\cdot(\frac1{r_2rw^2}-\frac r{r_2^2})\geq0$ holds everywhere on $[0,\frac14r_1]$. Using this observation, we continue the estimate \eqref{eq-cutoff:0to18p1}:
    \begin{align}
        -\frac{h''}h-\frac{2h'}{rh}-(u')^2 &\geq \frac{c_1}{1-c_1\psi}\Big[\frac{\zeta}{r^2w^2}+\frac{2\zeta}{r^2w^3}+3\frac{1-\zeta}{r_2}\Big]-\frac{c_2^2}{r^2w^2}\zeta^2 \\
        &\geq (c_1-c_2^2)\frac{\zeta}{r^2w^2}+c_1\frac{2\zeta}{r^2w^3}+3c_1\frac{1-\zeta}{r_2}\\
        {}&\geq c_1\frac{2\zeta}{r^2w^3}+3c_1\frac{1-\zeta}{r_2}, \label{eq-cutoff:interval1_term1}
    \end{align}
    where we have used the facts $\psi\geq0$ and $\zeta\leq1$ in the middle line, and the inequality $c_2\leq c_1\leq\sqrt{c_1}$ from item 2 of Lemma \ref{lemma-cutoff:c1c2r2} in the last line. Next, we control the last term in \eqref{eq-cutoff:reduction_to_h}. Using $c_1\psi\leq\frac12$ from Lemma \ref{lemma-cutoff:psi}, and the fact that $1-\zeta=0$ when $r\geq r_2$, we have
    \begin{align}
            2|k|r\frac{|h'|}{h}
            &= \frac{2|k|r}{1-c_1\psi}\Big|\frac{c_1\zeta}{rw^2}+c_1(1-\zeta)\frac r{r_2}\Big| \\
            &\leq 4|k|r\big[\frac{c_1\zeta}{rw^2}+c_1(1-\zeta)\big] \\
            &= 4|k|\frac{c_1\zeta}{w^2}+4|k|r_1c_1(1-\zeta),\qquad r\in[0,\frac14r_1] \label{eq-cutoff:lastterm18p1}.
    \end{align}
    Using $\frac1{r^2w}>4|k|$ from item 3 of Lemma \ref{lemma-cutoff:r1} and $r_1r_2\leq r_1^2\leq\frac1{2|k|}$ from item 1 of Lemma \ref{lemma-cutoff:r1}, we continue the estimate \eqref{eq-cutoff:lastterm18p1} to find
    \begin{equation}\label{eq-cutoff:interval1_term2}
        2|k|r\frac{|h'|}{|h|} \leq \frac{c_1\zeta}{r^2w^3}+2c_1\frac{1-\zeta}{r_2}.
    \end{equation}
    Finally, combining \eqref{eq-cutoff:reduction_to_h}, \eqref{eq-cutoff:interval1_term1}, \eqref{eq-cutoff:interval1_term2} we obtain
    \begin{equation}\label{eq-cutoff:interval1_est1}
        e^{2u}\Big(-\frac{f''}f-(u')^2\Big)\geq e^{2u}\Big(k+\frac{c_1\zeta}{r^2w^3}+c_1\frac{1-\zeta}{r_2}\Big),\qquad \qquad r\in [0,\frac14r_1].
    \end{equation}
    
    To complete the desired estimate \eqref{eq-cutoff:scalar_curvature} on $[0,\frac14 r_1]$, we divide the remaining work into three cases:

    {\noindent\bf{Case i:}} Assume $k\leq0$. In this case, we drop the last two terms in \eqref{eq-cutoff:interval1_est1} and use the fact that $u\leq0$ to obtain
    \begin{equation}\label{eq-cutoff:interval1_result1}
        e^{2u}\big(-\frac{f''}f-(u')^2\big)\geq e^{2u}k\geq k.
    \end{equation}
    
    {\noindent\bf{Case ii:}} Assume $k>0$ and $\zeta\leq\frac12$. Note that this implies $r<r_2$. For this range of $r$, we have the bound $e^u\geq w^{-c_2}\geq\frac1{\log(1/r_2)}$ from \eqref{eq-cutoff:u_bound}. Using this fact, 
    we estimate \eqref{eq-cutoff:interval1_est1} by dropping the first two terms:
    \begin{align}
        e^{2u}\Big(-\frac{f''}f-(u')^2\Big)&\geq e^{2u}c_1\frac{1-\zeta}{r_2} \\
        &\geq(\log\frac1{r_2})^{-2}\frac{c_1}{2r_2}=\frac{1}{2\sqrt{r_2}\log^2(1/r_2)}\cdot\frac{c_1}{\sqrt{r_2}}.\label{eq-cutoff:interval1_result2p}
    \end{align}
    Now apply the smallness conditions $\frac1{\sqrt{r_2}\log(1/r_2)^2}\geq100(|k|+1)$ from item 3 of Lemma \ref{lemma-cutoff:r1} and $r_2\leq c_1^2$ from item 3 of Lemma \ref{lemma-cutoff:c1c2r2} to continue \eqref{eq-cutoff:interval1_result2p}, concluding
    \begin{equation}
        e^{2u}\Big(-\frac{f''}f-(u')^2\Big)\geq 50(|k|+1).\label{eq-cutoff:interval1_result1}
    \end{equation}

{\noindent\bf{Case iii:}} Assume $k>0$ and $\zeta>\frac12$. We estimate \eqref{eq-cutoff:interval1_est1} by dropping the last term and using $e^u\geq\big(\frac{\log(1/r_1)}{w}\big)^{c_2}\geq\frac1w$ from \eqref{eq-cutoff:u_bound} with $\frac1{rw^5}\geq 2k$ from item 3 of Lemma \ref{lemma-cutoff:r1}
\begin{equation}\label{eq-cutoff:interval1_result2}
    e^{2u}\Big(-\frac{f''}f-(u')^2\Big)\geq k\Big(\frac{\log(1/r_1)}{w}\Big)^{2c_2}+\frac{c_1}{2r^2w^5}\geq k\frac{\log(1/r_1)^{2c_2}}{w^{2c_2}}+\frac{kc_1}{r}.
\end{equation}
To proceed, notice that the function $q(r)=k\frac{\log(1/r_1)^{2c_2}}{\log(1/r)^{2c_2}}+\frac{kc_1}{r}$ is decreasing on $[0,r_1]$. Indeed, using $c_2\leq c_1$, $r\leq r_1$, and $\frac1r>1>\frac2{\log(1/r_1)}$ we have
\begin{equation}
    q'(r)=\frac{2c_2k\log(1/r_1)^{2c_2}}{r\log(1/r)^{1+2c_2}}-\frac{kc_1}{r^2}\leq\frac{c_1k}{r}\big(\frac2{\log(1/r_1)}-\frac1r\big)<0,
\end{equation}
where we have used the fact that $r\mapsto\log(1/r)$ is decreasing.
It follows that \eqref{eq-cutoff:interval1_result2} implies
\begin{equation}\label{eq-cutoff:interval1_result3}
    e^{2u}\Big(-\frac{f''}f-(u')^2\Big)\geq q(r)\geq q(r_1)=k+\frac{kc_1}{r_1}>k.
\end{equation}

Inspecting the inequalities \eqref{eq-cutoff:interval1_result1}, \eqref{eq-cutoff:interval1_result2}, and \eqref{eq-cutoff:interval1_result3}, we have established the desired scalar curvature lower bound for $r\in[0,\frac14r_1]$. In light of estimate \eqref{eq-cutoff:scal_interval2} on the complementary interval, this proves condition (I) of Theorem \ref{t:T3tube}.

\vspace{12pt}

It remains to verify conditions (V), (VI), and (VII) of Theorem \ref{t:T3tube}.

\textbf{Distance estimate.}
By the remark following \eqref{eq-cutoff:u_bound}, we have $e^{-u(r)}\leq \log(1/r)^{c_2}$. Using the bound $c_2\leq r_1$ from item 2 of Lemma \ref{lemma-cutoff:c1c2r2}, this gives $e^{-u(r)}\leq \log(1/r)^{r_1}$. We apply this observation with an integration by parts to bound the distance:
\begin{align}
    \dist_{\G}(\{0\}\times\RR,\p D^2\times\RR)=\int_0^{r_1}e^{-u(r)}\,dr &\leq \int_0^{r_1}\log(\frac1r)^{r_1}\,dr \\
    &= r_1\log(\frac1{r_1})^{r_1}+r_1\int_0^{r_1}\log(\frac1r)^{r_1-1}\,dr.
\end{align}
Using the facts $(\log\frac1x)^x<2$ for all $x>0$ and $\log\frac1{r_1}>1$, we estimate the previous line to find
\begin{equation}\label{eq-cutoff:distest}
    \dist_{\G}(\{0\}\times\RR,\p D^2\times\RR) \leq 2r_1+r_1^2=3r_1.
\end{equation}
This establishes condition (VI).

\textbf{Volume estimate.} The volume in question is computed by integrating $e^{-u}f$. Since $h\leq1$, we have $f(r)\leq\sn_k(r)$. It follows that
\begin{align}
    \int_0^{r_1} e^{-u(r)}f\,dr &\leq \int_0^{r_1}\sn_k(r)e^{-u(r)}\,dr\leq 2r_1\cdot\int_0^{r_1}e^{-u(r)}\,dr\leq 6r_1^2,
\end{align}
where we used \eqref{eq-cutoff:sn} in the second to last inequality and the distance estimate \eqref{eq-cutoff:distest} in the final inequality.

\textbf{Mean convexity condition.} Recall from \eqref{e:meang} that the mean curvature of constant $r$ slices is $e^u\frac{f'}f$. First we compute
\begin{align}
    f' &= \cn_k(r)\cdot h+\sn_k(r)\cdot h' \\
    &= \cn_k(r)\cdot h+\sn_k(r)\big[\frac{c_1}{r_1}\eta'\big(\frac r{r_1}\big)\psi+c_1\eta\big(\frac r{r_1}\big)\zeta\big(\frac r{r_2}\big)\frac1{rw^2}+c_1\eta\big(\frac r{r_1}\big)\big(1-\zeta\big(\frac r{r_2}\big)\big)\frac r{r_2}\big].
\end{align}
Next, we drop the last two terms and apply $|\eta'|\leq4$ from \eqref{eq-cutoff:etazeta} to find
\begin{equation}
    f'\geq\cn_k(r)\cdot h-\sn_k(r)\cdot 4\cdot\frac12\geq \frac12\cn_k(r)-2\sn_k(r)\cdot ,\label{eq-cutoff:meanconv1}
\end{equation}
where we have made use of $\psi\leq\frac12$ (hence $h\geq\frac12$) from Lemma \ref{lemma-cutoff:psi} in the last inequality.
Meanwhile, the smallness conditions \eqref{eq-cutoff:sn} and item 1 of Lemma \ref{lemma-cutoff:r1} imply
\begin{equation}
    \cn_k(r)-4\sn_k(r)\geq\sn_k(r)\Big(\frac1r-|k|r-4\Big)\geq\sn_k(r)\Big(100(1+|k|)-|k|-4\Big)\geq 90\sn_k(r).
\end{equation}
Combining this with \eqref{eq-cutoff:meanconv1} shows $f'>0$ and the mean convexity statement follows.
\end{proof}

\section{The proofs of Theorems \ref{t:T3} and \ref{t:S2S1}}\label{sec:mainthm}

Having established Theorem \ref{t:T3tube}, we are now prepared to present the main examples and prove Theorems \ref{t:T3} and \ref{t:S2S1}. These examples are created by modifying the product $\Sigma\times S^1$ of a constant curvature surface $(\Sigma,g)$ with a circle near a curve $\{*\}\times S^1$, replacing a tubular neighborhood with a drawstring of the same curvature as $(\Sigma,g)$. The only difference between the two Theorems is the choice of the surface; we take $\Sigma$ as a flat torus for Theorem \ref{t:T3} and $\Sigma$ as the unit sphere in Theorem \ref{t:S2S1}. Due to the similarity of the proofs, we will provide the construction for Theorem \ref{t:T3} in detail, then describe the alterations necessary to prove Theorem \ref{t:S2S1}.

\begin{proof}[Proof of Theorem \ref{t:T3}]\label{sec:lp} 
Let $(T^3,\boldsymbol{\delta})$ denote the product of three circles, each $2\pi$ in length, and let $\gamma$ denote the vertical cycle $\gamma=\{(0,0)\}\times S^1$. For an integer $i>0$, we apply Theorem \ref{t:T3tube} with $k=0$, $\varepsilon=\delta=1/i$, and $r_0<1/i^2$ small enough to ensure that the volume $100 V_0$ appearing in property (VII) is less than $1/i^4$, which is possible according to Remark \ref{rem:epsilonpipe}. Let $A_i=(D^2\times S^1,{\widehat{\G}}_i)$ denote the resulting $\varepsilon$-drawstring where, in the notation of Theorem \ref{t:T3tube}, we have made the identification $t\sim t+2\pi$. Let $r_1\leq r_0$ be the radius given by condition (II); note that $r_1$ depends implicitly on $i$. Remove the $2r_1$-neighborhood of $\gamma$ from $T^3$ and glue in $(A_i,{\widehat{\G}}_i)$ to obtain the sequence of manifolds $\{(T^3,\G_i)\}_{i=1}^\infty$, which are smooth by the first property of Theorem \ref{t:T3tube} and with scalar curvature bounded from below by $-1/i$. By construction, $\G_i$ have a warped product form. Since $100 V_0<1/i^4$, the volume of $\G_i$ converges to $(2\pi)^3$. Moreover, by property (VI), the distance between any two points $p,q$ in $(T^3,\G_i)$ can be bounded above by their flat distance $d_\delta(p,q)+2/i^2$ by first pushing them outside $A_i$, which establishes the desired diameter upper bound. The diameter lower bound follows by considering points at least distance 1 from $A_i$.

In order to conclude the desired convergence to the space that results from pulling $\gamma$ to a point, we will apply Theorem \ref{t:scrunch}. To proceed, we verify that the sequence $\{N_i\}_{i=1}^\infty=\{(T^3,\G_i)\}_{i=1}^\infty$ satisfies conditions $(i), (ii),$ and $(iii)$ of Definition \ref{def:scrunch}. In the language of that definition, we take
\begin{equation}
    U_i=\left(B^{\boldsymbol{\delta}}(\gamma,1/i)\setminus B^{\boldsymbol{\delta}}(\gamma,2r_1)\right)\cup A_i,
\end{equation}
$\varepsilon_i=\delta_i=1/i$, and $H_i=3/i$, shown in Figure \ref{fig:torus}. Property $(i)$ is evident. To confirm property $(ii)$, we calculate 
\begin{equation}
    |U_i|_{\G_i}\leq2\pi(1/i^2-(2r_1)^2)+100V_0=2\pi/i^2+O(i^{-4}).
\end{equation}
Finally, by considering paths that travel to and along $\gamma$, the diameter of $U_i$ is bounded by $|\gamma|_{\G_i}+2d_{\G_i}(\gamma,\p U_i)$. This length converges to zero as $i\to\infty$ by conditions (IV) and (VI). It follows that $\{(T^3,\G_i)\}_{i=1}^\infty$ converges to the desired pulled string space in the volume-preserving intrinsic flat sense.

The last step is to establish the uniform lower bound of the $\mathrm{minA}$ invariant of $(T^3,\G_i)$. Suppose $S\subset T^3$ is a closed smoothly embedded surface which is minimal with respect to the metric $\G_i$. We first claim that some portion of $S$ must lie outside the distance neighborhood $B^{\boldsymbol{\delta}}(\gamma,1)$. Indeed, if $S$ were contained within $B^{\boldsymbol{\delta}}(\gamma,1)$, there would exist an $r_*< \pi$ so that $S$ lies within $\overline{B^{\G_i}(\gamma,r_*)}$ and has a tangential intersection with the cylinder $C_{r_*}=\partial B^{\G_i}(\gamma,r_*)$. If $r_*<2r_1$, then property $(V)$ of Theorem \ref{t:T3tube} implies $C_{r*}$ is strictly mean convex (with respect to the normal pointing away from $\gamma$). On the other hand, if $r_*\geq 2r_1$, then the metric $\G_i$ is Euclidean near $C_{r_*}$, and so $C_{r_*}$ is again strictly mean convex. In either case, the one-sided tangential intersection with $S$ would violate the maximum principle for minimal surfaces. It follows that $S$ must exit $B^{\boldsymbol{\delta}}(\gamma,1)$, and so there exists a $p\in S$ with the property that, for $i>1$, $B^{\G_i}(p,\frac12)$ avoids the non-flat region $A_i$. By the classical monotonicity formula for minimal surfaces \cite{CM}, we may conclude that $B^{\G_i}(p,\frac12)\cap S$ has area at least $\pi/4$. The $\minA$ lower bound follows, finishing the proof of Theorem \ref{t:T3}.
\end{proof}

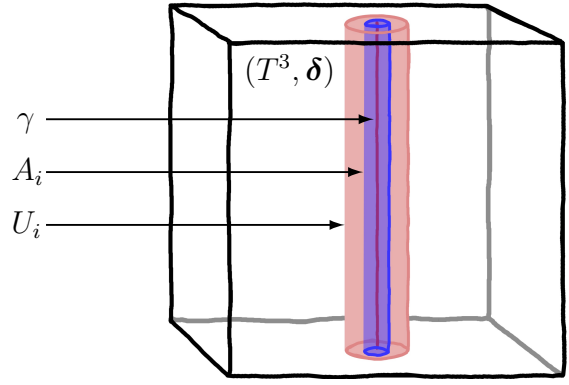
\begin{figure}[h]
    \begin{center}
        \includegraphics[totalheight=5cm]{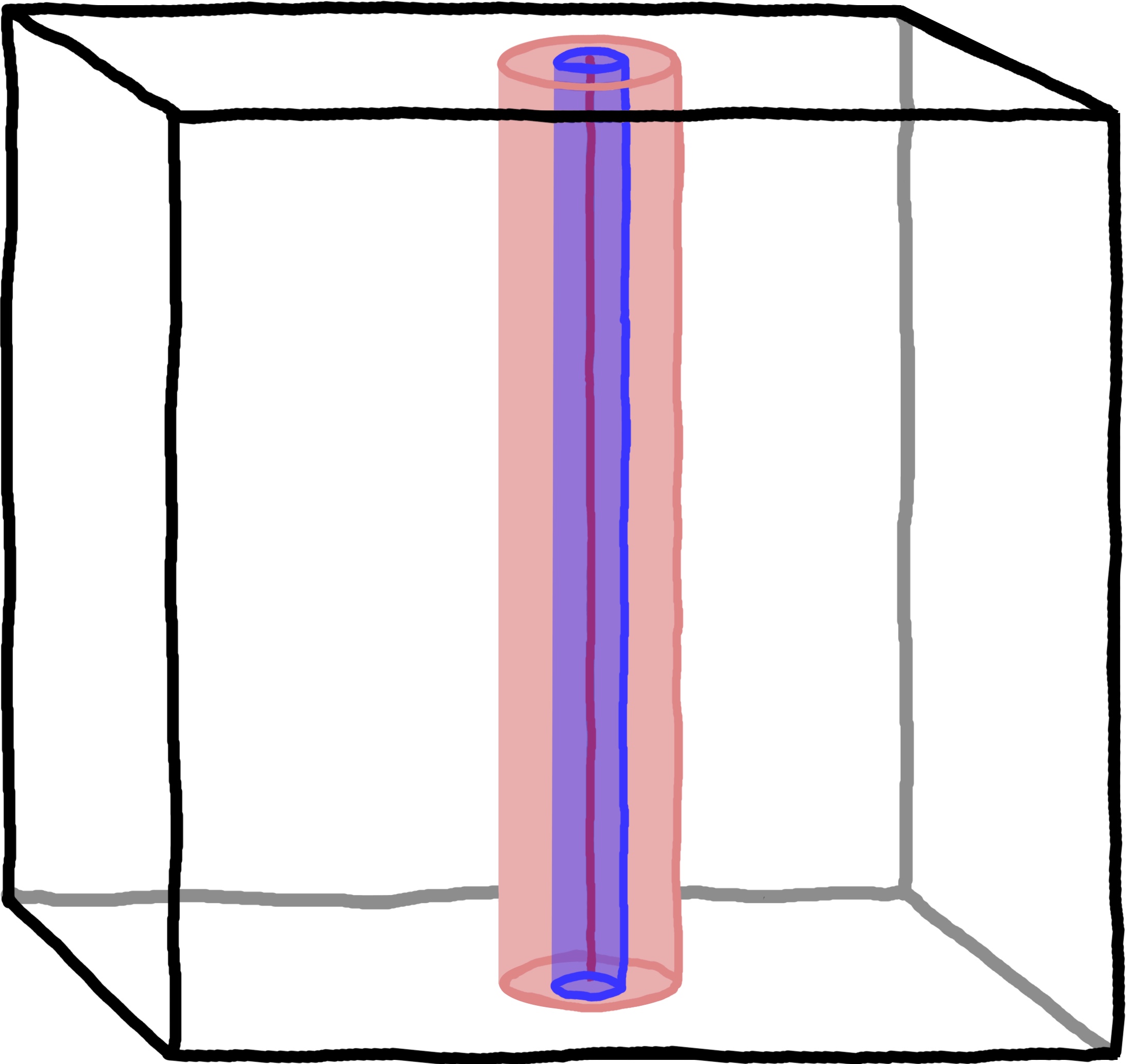}
    \end{center}
    \begin{picture}(0,0)(221,6)
        \put(175,135){$\large{(T^3,\boldsymbol{\delta})}$}
        \put(89,117){{\large{$\gamma$}}}
        \linethickness{0.25mm}
        \put(100,120){\vector(10,0){124}}
        \put(87,97){{\large{$A_i$}}}
        \put(100,100){\vector(10,0){120}}
        \put(87,77){{\large{$U_i$}}}
        \put(100,80){\vector(10,0){113}}
    \end{picture}
    \setlength{\abovecaptionskip}{-10pt}
    \captionsetup{margin=.8cm}
    \caption{A cube is shown, representing the $3$-tori in the proof of Theorem \ref{t:T3} upon identifying opposing faces. The region $U_i$ and the $\varepsilon$-drawstring $A_i$ are depicted in red and blue.}
    \label{fig:torus}
\end{figure}

\begin{proof}[Proof of Theorem \ref{t:S2S1}]
This construction is nearly identical to the one proving Theorem \ref{t:T3}, so we will be brief. Let $N_0=(S^2\times S^1,\G_0)$ be the product of the standard sphere with a circle of length $2\pi$, and $\gamma=\{x_0\}\times S^1$ be a fixed circle. Now we apply Theorem \ref{t:T3tube} with the choice $k=1$, and the same parameters $\epsilon=\delta=1/i$, and $r_0<1/i^2$ small enough so that $100V_0<1/i^4$. By removing a $2r_1$-neighborhood of $\gamma$ from $N_0$ and gluing in the resulting drawstring, we obtain a sequence $N_i=(S^2\times S^1,\G_i)$. According to our choice of parameters, $N_i$ has scalar curvature at least $2-1/i$, as well as controlled diameter and volume. An argument identical to the one provided in the proof of Theorem \ref{t:T3} shows that $N_i$ converges to the desired pulled string space in intrinsic flat sense.

To demonstrate the minA lower bound, suppose $S\subset N_i$ is a closed smoothly embedded minimal surface. Leveraging condition (V) of Theorem \ref{t:T3tube} and the geometry of the product $N_0$, one can show that the cylinders $\partial B^{\G_i}(\gamma,r)$ are strictly mean convex for $r<\frac12$. It follows that $S$ cannot lie entirely within $B^{\G_i}(\gamma,\frac12)$. Choose a point $p\in S\cap(N_i\setminus B^{\G_i}(\gamma,\frac12))$ and note, for $i>1$, that $B^{\G_i}(p,\frac14)$ is isometric to a ball of radius $\frac14$ in the product manifold $N_0$. We now apply a more general monotonicity formula \cite[(7.5)]{CM} to find $|S|_{\G_i}\geq A_0$, where $A_0$ is a constant depending only on the geometry of $N_0$.
\end{proof}


\section{\texorpdfstring{$W^{1,p}$}{W^(1,p)} Convergence of Metric Tensor}\label{sec:lp}

\subsection{Preliminary facts}

The following facts follow from direct computation. To see the scalar curvature formula within the following lemma's second item, see Proposition \ref{prop:curvatures} and its proof.

\begin{lemma}\label{lemma-aux:scalar_formula} 
    Let $g$ and $\tilde{g}$ be Riemannian metrics on a compact surface $\Sigma$ with Gauss curvatures $K$ and $\tilde K$, respectively. Given two smooth functions $u$ and $\varphi>0$ on $\Sigma$, the following hold:
	\begin{enumerate}
	    \item If $(M=\Sigma\times S^1,\G)$ has the form $\G=\g+\varphi^2dt^2$, then $R_\G=2(\K-\frac{\Delta\varphi}{\varphi})$. Moreover, $\diam(\Sigma,\g)\leq\diam(M,\G)$.
            \item If $\G$ has the form $\G=e^{-2u}\cg+e^{2u}dt^2$, then $R_\G=2e^{2u}(\cK-|\cD u|^2)$. Moreover, a curve $\gamma\subset\Sigma$ is a $\cg$-geodesic if and only if $\gamma\times S^1$ is a minimal surface for $(M,\G)$.
	\end{enumerate} 
\end{lemma}

As an immediate consequence of the first item in Lemma \ref{lemma-aux:scalar_formula}, we have the following fact about the spectrum of the operator $-\Delta+K$ on a closed surface.

\begin{cor}\label{cor-aux:scal_spectral}
    Suppose $(\Sigma,g)$ is a closed Riemannian surface with a smooth function $\varphi>0$ and let $\lambda>0$. If $\G=\g+\varphi^2dt^2$ has scalar curvature $R_{\G}\geq-2\lambda$, then $\g$ satisfies $\lambda_1(-\Delta+\K)\geq-\lambda$, which means that any smooth function $\psi$ satisfies
    \begin{equation}\label{eq-aux:spectral}
       \int_\Sigma\big(|\D\psi|^2+K\psi^2\big)\dA\geq-\lambda\int_\Sigma\psi^2\dA.
    \end{equation}
\end{cor}
\begin{proof}
    Notice that $\varphi$ is positive and satisfies $\Delta\varphi\leq(K+\lambda)\varphi$. Then the result follows from mini-max characterization of the principle eigenvalue of an elliptic operator, $\lambda_0(L)=\sup_{u>0}\inf_\Sigma \frac{Lu}{u}$.
\end{proof}

We need the following geometric inequalities for metrics that satisfy \eqref{eq-aux:spectral}.

\begin{lemma}\label{lemma-aux:geometric_ineq}
    Let $(\Sigma,g)$ be a Riemannian $2$-torus satisfying $\lambda_1(-\Delta+K)\geq-\lambda$ for $\lambda>0$, $\diam(\Sigma,g)\leq D$, and $|\Sigma|_g\geq A_0$. Then the following hold:
    \begin{enumerate}
        \item The $\pi_1$-systole of $(\Sigma,g)$, defined as the length of the shortest non-contractible closed geodesic, has a positive lower bound that depends only on $\lambda,D,A_0$.
        \item The area of $\Sigma$ has an upper bound that depends only on $\lambda,D$.
        \item The isoperimetric ratio, defined as
        \begin{equation}
            \IN(\Sigma,\g):=\inf\left\{\frac{|\p\Omega|_g^2}{\min\{|\Omega|_g,|\Sigma\setminus\Omega|_g\}}:\ \Omega\subset\Sigma\text{ is a smooth domain}\right\},
        \end{equation}
        has a uniform lower bound depending only on $\lambda,D,A_0$.
    \end{enumerate}
\end{lemma}
\begin{proof}
    To establish item $1$, let $\gamma\subset\Sigma$ be a closed geodesic representing a non-trivial class in $\pi_1(\Sigma)$. Denote $\rho=\max_{x\in\Sigma}d(x,\gamma)$, so $\rho\leq D$. For almost every $s\in(0,D]$, the equi-distance set $\gamma(s)=\{x\in\Sigma: d(x,\gamma)=s\}$ is a piecewise smooth curve \cite{ST1, ST2}. Let $L(s)=|\gamma(s)|$ be its length. In general, $L(s)$ is in general not continuous in $s$. However, due to \cite[Theorem 2.2]{ST2} there is a function $J(s)$ with vanishing derivative almost everywhere so that $L(s)+J(s)$ is absolutely continuous. Hence, $L(s)$ is almost everywhere differentiable. For a positive non-increasing function $\phi:[0,\rho]\to[0,\infty]$, an argument using the spectral condition (\ref{eq-aux:spectral}) with the choice $\psi(x)=\phi(d(x,\gamma))$ yields
    \begin{equation}\label{eq-aux:fund_eq_nonsep}
	\begin{aligned}
		&\int_0^\rho (\phi')^2L\,ds+2\int_0^\rho L\phi\phi''\,ds-\lambda\int_0^\rho L\phi^2\,ds \\
		&\qquad\qquad\qquad\qquad\qquad\qquad\leq 2\pi\chi(\Sigma)\phi(\rho)^2-4\,|\gamma|\,\phi(0)\phi'(0).
	\end{aligned}
    \end{equation}
    For the sake of exposition, we postpone the proof of \eqref{eq-aux:fund_eq_nonsep} shortly afterward. Taking $\phi(s)=e^{-\sqrt\lambda s}$ in \eqref{eq-aux:fund_eq_nonsep}, a basic computation shows
    \begin{equation}\label{e:systole}
	4\sqrt\lambda|\gamma|
        \geq 2\lambda\int_0^\rho Le^{-2\sqrt\lambda s}\,ds
        \geq 2\lambda e^{-2\sqrt\lambda\rho}|\Sigma|_g
        \geq 2\lambda e^{-2\sqrt\lambda D}A_0
    \end{equation}
    where we have used the coarea formula in the second inequality above. Rearranging \eqref{e:systole} gives the lower bound on the systole.
    
    
    The area upper bound in item $2$ is obtained from an analogue of \eqref{eq-aux:fund_eq_nonsep} for separating curves. To state the relevant inequality, some notations are required. Suppose $\gamma\subset \Sigma$ is a smooth connected closed curve and that $\gamma$ separates $\Sigma$ into two regions $\Omega^+,\Omega^-$. Define the signed distance function
    \begin{equation}
    \rho(x)=\begin{cases}
          d(x,\gamma)&\text{ for }x\in\Omega^+, \\
         -d(x,\gamma)&\text{ for }x\in\Omega^-.
    \end{cases}
    \end{equation}
    Let $\rho^+$ and $\rho^-$ denote the extremal distances $\rho^+=\max_{x\in\Sigma}\rho(x)$, $\rho^-=|\min_{x\in\Sigma}\rho(x)|$, and let $L(s)$ be the length of the curve $\{x\in\Sigma:\rho(x)=s\}$.
    If $\phi:[-\rho^-,\rho^+]\to[0,\infty]$ is a nonnegative continuous function satisfying $\phi'\leq0$ when $s>0$ and $\phi'\geq0$ when $s<0$, \cite[Lemma 2.3]{Xu_2022} gives the following:
    \begin{equation}\label{eq-aux:fund_eq_sep}
	\begin{aligned}
		&\int_{-\rho^-}^{\rho^+}(\phi')^2L\,ds
			+2\int_{-\rho^-}^{\rho^+}L\phi\phi''\,ds
			-\lambda\int_{-\rho^-}^{\rho^+}L\phi^2\,ds \\
		&\qquad\qquad\leq
            2\pi\Big[\chi(\Omega^+)\phi(\rho^+)^2+\chi(\Omega^-)\phi(\rho^-)^2\Big]
		+2|\gamma|\,\phi(0)\Big[\phi'_-(0)-\phi'_+(0)\Big],
	\end{aligned}
    \end{equation}
    where $\phi'_\pm$ are the one-sided derivatives of $\phi$. To prove item $2$, we consider $\gamma=\p B(x,\varepsilon)$ for $\varepsilon$ small enough to ensure $B(x,\varepsilon)$ is contractible. Take $\Omega^+=\Sigma\setminus B(x,\epsilon)$ and $\Omega^-=B(x,\epsilon)$. We apply \eqref{eq-aux:fund_eq_sep} with the function $\phi$ with $\phi(s)=1$ when $s\leq0$ and $\phi(s)=e^{-\sqrt\lambda s}$ when $s\geq0$ to obtain
    \begin{equation}
        -\lambda\int_{-\rho^-}^0Lds+2\lambda\int_0^{\rho^+}L e^{-2\sqrt{\lambda}s}ds\leq 2\pi(1-e^{-2\sqrt\lambda\rho^+})+2|\partial B(x,\varepsilon)|\sqrt{\lambda}.        
    \end{equation}
    Applying the coarea formula to the left side results in
    \begin{equation}\label{e:item2eq}
        -\lambda|B(x,\varepsilon)|+2\lambda e^{-2\sqrt{\lambda}D}|\Sigma\setminus B(x,\varepsilon)|
        \leq 2\pi+2|\partial B(x,\varepsilon)|\sqrt{\lambda}.
    \end{equation}
    Taking $\varepsilon\to0$ in \eqref{e:item2eq} shows $2\lambda e^{-2\sqrt\lambda D}|\Sigma|\leq 2\pi$, yielding the desired area bound.

 
    For the proof of item $3$, let $\Omega\subset\Sigma$, with $\gamma=\p\Omega$. It is sufficient to prove the case where both $\Omega$ and $\Sigma\setminus\Omega$ are connected, since the ratio appearing in the definition of $\mathrm{IN}(\Sigma,g)$ decreases when passing to connected components. Moreover, it suffices to consider the case where all the components of $\gamma$ are separating. Indeed, if any connected component of $\gamma$ is non-separating, then the result follows from items $1$ and $2$. Combined with the fact that all components of $\gamma$ separate, the connectedness of $\Omega$ and $\Sigma\setminus\Omega$ implies the connectedness of $\gamma$. Let $\Omega^+=\Omega$, $\Omega^-=\Sigma\setminus\Omega$. Since $\Omega_\pm$ are connected with a single boundary component, $\chi(\Omega^\pm)$ are either $1$ or $-1$. Due to the fact that $\chi(\Omega^+)+\chi(\Omega^-)=\chi(\Sigma)=0$, we can assume that $\chi(\Omega^+)=1$ and $\chi(\Omega^-)=-1$.
    
    Now denote $\rho^+=\max_{x\in\Omega^+}d(x,\gamma)$ and $\rho^-=\max_{x\in\Omega^-}d(x,\gamma)$.
    If $\rho^+\leq\rho^-$, then we choose
    \begin{equation}
        \phi(s)=\begin{cases}
	 e^{\sqrt\lambda s}&\text{ for }-\rho^-\leq s\leq0 \\
	 e^{-\rho^-\sqrt\lambda s/\rho^+}&\text{ for }0\leq s\leq\rho^+
    \end{cases}
    \end{equation}
    for which \eqref{eq-aux:fund_eq_sep} gives
    \begin{equation}\label{eq-lp:separating}
	2\lambda\int_{\rho^-}^0Le^{2\sqrt\lambda s}\,ds+2\lambda\frac{(\rho^-)^2}{(\rho^+)^2}\int_0^{\rho^+}Le^{-2\rho^-\sqrt\lambda s/\rho^+}\,ds\leq 4|\gamma|\sqrt\lambda\frac{\rho^-}{\rho^+}.
    \end{equation}
    Applying the coarea formula to the left side of \eqref{eq-lp:separating} and throwing away the first term shows that $|\gamma|\geq\frac{\sqrt\lambda}2\frac{\rho^-}{\rho^+}e^{-2\sqrt\lambda D}|\Omega^+|$. Returning to \eqref{eq-lp:separating} again, we obtain another inequality
    \begin{equation}
        2\lambda e^{-2\sqrt\lambda D}|\Sigma|\leq 4|\gamma|\sqrt\lambda\frac{\rho^-}{\rho^+}.
    \end{equation}
    Therefore,
    \begin{equation}
        \frac{|\gamma|^2}{|\Omega^+|}
        = \frac{|\gamma|}{|\Omega^+|}\cdot|\gamma|
        \geq \frac{\sqrt\lambda}2\frac{\rho^-}{\rho^+}e^{-2\sqrt\lambda D}|\gamma|
        \geq \frac\lambda4e^{-4\sqrt\lambda D}A_0.
    \end{equation}
    When $\rho^+\geq\rho^-$, the result follows with choosing
    \begin{equation}
        \phi(s)=\begin{cases}
	 e^{\rho^+\sqrt\lambda s/\rho^-}&\text{ for }-\rho^-\leq s\leq0 \\
	 e^{-\sqrt\lambda s}&\text{ for }0\leq s\leq\rho^+
        \end{cases}
    \end{equation}
    by the similar argument.
\end{proof}

The following work will complete the proof of Lemma \ref{lemma-aux:geometric_ineq}.

\begin{proof}[{\bf{Proof of \eqref{eq-aux:fund_eq_nonsep}}}]
    The argument is similar to the proof of \eqref{eq-aux:fund_eq_sep} given in \cite[Lemma 2.3]{Xu_2022} so we will be brief. Denote $d(x)=d(x,\gamma)$, $\rho=\max_\Sigma(d)$, and for $s>0$:
    \begin{align}
    \begin{split}
        & \gamma(s)=\{d=s\},\quad L(s)=|\gamma(s)|,\quad
        \Omega(s)=\{d<s\},\quad
        \chi(s)=\chi(\Omega(s)), \\
        &\ \ 
        K(s)=\int_{\gamma(s)}K\,dl,\quad
        G(s)=\int_0^s K(t)\,dt,\quad
        \Gamma(s)=2\pi\chi(s)-G(s).
    \end{split}
    \end{align}
    By \cite{Hartman, ST1, ST2} we have $L'(s)\leq\Gamma(s)$ in the sense of distributions, and when $\gamma(s)$ is a smooth curve, we have $L'(s)=\Gamma(s)$ classically, which follows from the Gauss-Bonnet formula. 
    
    Testing \eqref{eq-aux:spectral} with $\psi(x)=\phi(d(x))$ for the given positive function $\phi$ (with $\phi'\leq0$), we obtain
    \begin{equation}\label{eq-aux:aux1}
        \int_0^\rho \big(L(s)(\phi')^2+K(s)\phi^2\big)\,ds\geq-\lambda\int_0^\rho L(s)\phi^2\,ds
    \end{equation}
    Using integration by parts on the second term on the left side of \eqref{eq-aux:aux1}, we obtain
    \begin{align}
        \int_0^\rho K(s)\phi^2\,ds &=G(\rho)\phi(\rho)^2-G(0)\phi(0)^2-\int_0^\rho 2G(s)\phi\phi'\,ds \nonumber\\
        &= 2\pi\chi(\Sigma)\phi(\rho)^2\,+2\int_0^\rho\Gamma(s)\phi\phi'\,ds-4\pi\int_0^\rho\chi(s)\phi\phi'\,ds \nonumber\\
        &\leq 2\pi\chi(\Sigma)\phi(\rho)^2+2\int_0^\rho\Gamma(s)\phi\phi'\,ds \label{eq-aux:aux2}
    \end{align}
    where we have made use of the coarea formula with the Gauss-Bonnet formula in the second equality and the facts $\chi(s)\leq0$, $\phi(s)\geq0$, and $\phi'(s)\leq0$ in the last inequality. From the distributional inequality $L'\leq\Gamma$ one can show (see \cite{Xu_2022} for technical details) that
    \begin{align}
        \int_0^\rho\Gamma(s)\phi\phi'ds &\leq -\int_0^\rho L(s)(\phi\phi')'ds-L(0^+)\phi(0)\phi'(0) \nonumber\\
        &= -\int_0^\rho L(s)(\phi'^2+\phi\phi'')ds-2|\gamma|\phi(0)\phi'(0). \label{eq-aux:aux3}
    \end{align}
    Combining \eqref{eq-aux:aux1}, \eqref{eq-aux:aux2}, and \eqref{eq-aux:aux3} yields \eqref{eq-aux:fund_eq_nonsep}.
\end{proof}

The final piece of preparation is the following uniform Moser-Trudinger inequality, which follows from \cite[Theorem 3.3]{Xu_2022}.

\begin{lemma}\label{lemma-aux:uniform_MT}
    There is a constant $C>0$ so that the following holds: Let $(\Sigma,g)$ be a closed Riemannian surface satisfying $\IN(\Sigma)\geq\xi>0$. Then for any smooth function $u$ on $\Sigma$, we have
    \begin{equation}
        \int_\Sigma\exp\Big(\frac{\xi(u-\bar u)^2}{\int_\Sigma|\D u|^2\,dA}\Big)\,dA\leq C|\Sigma|,
    \end{equation}
    where $\bar u$ is the average value of $u$ on $\Sigma$. Moreover, for any $p\geq1$, we also have
    \begin{equation}\label{eq-prelim:uniform_MT}
        \int_\Sigma e^{2p|u-\bar u|}\,dA
        \leq C|\Sigma|\cdot\exp\Big(\frac{p^2}{\xi}\int_\Sigma|\D u|^2\,dA\Big).
    \end{equation}
\end{lemma}

\subsection{Proof of Theorem \ref{thm-intro:convergence_T3}}

To establish Theorem \ref{thm-intro:convergence_T3}, we make several preliminary arguments which are broken down into four claims. At the end of this section, these facts are combined to finish the proof. Denote $\Sigma=T^2$, so that $M_i\cong\Sigma\times S^1$. Recall that $\G_i$ takes the form
\begin{equation}
    \G_i=g_i+\varphi_i^2dt^2,
\end{equation}
where $g_i$ may be considered as metrics on $\Sigma$. 

Before beginning the proof, we make some preparatory observations and notations. According to the diameter and $\minA$ bounds and the submersion form of $\G_i$, $(\Sigma, g_i)$ all have diameter bounded above by $D_0$ and with area no smaller than $A_0$. By scalar curvature lower bound of $\G_i$, Corollary \ref{cor-aux:scal_spectral} and Lemma \ref{lemma-aux:geometric_ineq}, there is a constant $C_1$ depending only on $D_0$ and $A_0$ so that the following uniform geometric inequalities hold:
\begin{equation}\label{eq-lp:geom_ineq_gi}
    |\Sigma|_{g_i}\leq C_1,\qquad
    \IN(\Sigma,g_i)\geq 1/C_1.
\end{equation}
By reparameterizing the $t$-coordinate by $t\mapsto t||\varphi_i||_{L^2(g_i)}$, we may assume that $\int_\Sigma \varphi_i^2\dA_i=1$. Fixing this normalization, let $[0,T_i]$ denote the range of the $t$-coordinates appearing in $\G_i$. Notice that $T_i=\vol(M,\G_i)/\int_\Sigma\varphi_i\dA_i$. Denote $u_i=\log(\varphi_i)$, and $\cg_i=e^{2u_i}\g_i$, so that the metric is written as
\begin{align}
    \G_i=e^{-2u_i}\cg_i+e^{2u_i}dt^2,\quad 0\leq t\leq T_i.
\end{align}
Observe that $\cg_i$ has unit area. By Lemma \ref{lemma-aux:scalar_formula} we have
\begin{equation}\label{eq-lp:K>du^2}
	\cK_i\geq-\frac1ie^{-2u_i}+|\cD u_i|^2,
\end{equation}
where we used the notation $|\cD u_i|^2=|\D_{\cg_i}u_i|^2_{\cg_i}$ for brevity.

Integrating \eqref{eq-lp:K>du^2} with respect to $\cg$ and using the uniform area upper bound, we have
\begin{equation}\label{eq-lp:integral_du2}
    \int_\Sigma|\D u_i|^2\,dA_i
    = \int_\Sigma|\cD u_i|^2\cdA_i
    \leq i^{-1}|\Sigma|_{g_i}
    \leq i^{-1}C_1.
\end{equation}

\noindent{\bf{Claim 1:}} $\{T_i\}_{i=1}^\infty$ are uniformly bounded from above and below.

\begin{proof}[Proof of Claim 1]
    For any curve $\eta$ in $\Sigma$, the $\G_i$-area of the surface $\eta\times S^1$ is equal to $T_i\int_\gamma e^{u_i}\,dl_{g_i}=T_i|\eta|_{\cg_i}$. By the normalization $|\Sigma|_{\cg_i}=1$ and Loewner's systolic inequality \cite[Theorem 4.1]{metricstructures}, there is a closed $\cg_i$-geodesic $\eta\subset\Sigma$ with $|\eta|_{\cg_i}\leq \tfrac{2}{\sqrt{3}}|\Sigma|_{\cg_i}^{1/2}=\tfrac{2}{\sqrt{3}}$. Since $\eta\times S^1\subset M_i$ is minimal by part 2 of Lemma \ref{lemma-aux:scalar_formula}, the $\minA$ condition implies
    \begin{equation}
	A_0\leq T_i|\eta|_{\cg_i}\leq \tfrac{2}{\sqrt{3}}T_i,
    \end{equation} 
    which gives the uniform lower bound on $T_i$. On the other hand, the second part of Lemma \ref{lemma-aux:uniform_MT} and the isoperimetric inequality \eqref{eq-lp:geom_ineq_gi} imply there is a constant $C_2$ depending only on $D_0$ and $A_0$ so that
    \begin{equation}
        \begin{aligned}
            \int_\Sigma e^{-u_i}\dA_i
            &= \int_\Sigma e^{-u_i}\dA_i\cdot\left(\int_\Sigma e^{2u_i}\dA_i\right)^{1/2} \\
            &= \int_\Sigma e^{\bar{u_i}-u_i}\dA_i\cdot\left(\int_\Sigma e^{2(u_i-\bar{u_i})}\dA_i\right)^{1/2}
            \leq C_2,
        \end{aligned}
    \end{equation}
    hence
    \begin{equation}
        \begin{aligned}
	V_0 \geq T_i\int_\Sigma e^{u_i}\dA_i
        \geq \frac{T_i\,|\Sigma|_{\g_i}^2}{\int_\Sigma e^{-u_i}\,dA_i}
        \geq \frac{T_iA_0^2}{C_2}.
    \end{aligned}
    \end{equation}
    This shows the upper bound.
\end{proof}

The reference flat metric $\G_\infty$ in the statement of Theorem \ref{thm-intro:convergence_T3} is found by means of the Uniformization Theorem. Namely, we may consider smooth functions $\psi_i$, defined up to an additive constant, such that $\cg_i=e^{2\psi_i}\fg_i$ and $\fg_i$ are flat metrics for $i=1,2,\dots$. To specify $\psi_i$ uniquely, we demand that the systole of $(\Sigma,\fg_i)$ is equal to $1$. With a suitable choice of the generators of the lattice, we may identify $(\Sigma,\fg_i)$ with a quotient $\RR^2/\Gamma_i$ for some lattice $\Gamma_i=\text{span}_\mathbb{Z}(1,z_i)$, where $z_i\in\mathbb{C}$ satisfies $|\Ree(z_i)|\leq1/2$, $\Imm(z_i)>0$, and $|z_i|\geq1$. Note that this implies $\Imm(z_i)\geq\sqrt 3/2$. \\

{\noindent\bf{Claim 2:}} $\Imm(z_i)$ is uniformly bounded.

\begin{proof}[Proof of Claim 2]
    Similar to the first part of the proof of Claim 1, combining the $\minA$ lower bound with Lemma \ref{lemma-aux:scalar_formula} shows that the systole of $(\Sigma,\cg_i)$ is bounded from below by $A_0/T_i$. To proceed, we must recall the notion of extremal length: Given a family $S$ of curves on $(\Sigma,\fg_i)$, the extremal length associated to $S$ is defined by
    \begin{equation}\label{eq-lp:def_EL}
        EL_{\fg_i}(S)=\sup\Big\{\inf_{\gamma\in S}\frac{\big(\int_\gamma \rho\,dl_{\fg_i}\big)^2}{\int_\Sigma\rho^2\fdA_i}: \rho\in L^\infty(\Sigma),\; \rho>0\Big\}.
    \end{equation}
    We take $S$ consisting of all the smooth closed curves in $(\Sigma,\fg_i)$ that are homologous to the image of the closed segment from $(0,0)$ to $(1,0)$ within $\RR^2/\Gamma_i$. It follows from Loewner's argument that $EL_{\fg_i}(S)=\Imm(z_i)^{-1}$. On the other hand, choosing the weight function $\rho=e^{\psi_i}$ in \eqref{eq-lp:def_EL}, we find
	\begin{equation}
	    EL_{\fg_i}(S)\geq\inf_{\gamma\in S}\frac{|\gamma|_{\cg_i}^2}{|\Sigma|_{\cg_i}}\geq \left(\frac{A_0}{T_i}\right)^2.
	\end{equation}
    Hence $\Imm(z_i)\leq T_i^2/A_0^2$, which finishes the proof in light of Claim 1.
\end{proof}

Next, we prepare to show that $\psi_i$ and $u_i$ are almost constant. Written in terms of the flat metrics $\fg_i$, the Gauss curvature bound \eqref{eq-lp:K>du^2} is equivalent to the following supersolution condition:
\begin{equation}\label{eq-lp:supersol}
	\fDelta\psi_i\leq i^{-1}e^{2\psi_i-2u_i}-|\fD u_i|^2.
\end{equation}
By the conformal invariance of Dirichlet energy, \eqref{eq-lp:integral_du2} gives $\int_\Sigma|\fD u_i|^2\fdA_i=\int_\Sigma|\cD u_i|^2\cdA_i\leq C_1i^{-1}$. Moreover, $\int_\Sigma e^{2\psi_i-2u_i}\fdA_i=|\Sigma|_{g_i}\leq C_1$. Using these observations, integrating \eqref{eq-lp:supersol} gives $|\fDelta\psi_i|_{L^1(\fg_i)}\leq 2C_1i^{-1}$. In turn, we may apply an adapted Brezis-Merle's Lemma to flat tori (presented in Lemma \ref{lemma-app:BM}) with Claim 2 to find a universal constant $C_3$ so that
\begin{equation}\label{eq-lp:brezis-merle_for_psi}
    \int_\Sigma\exp\left(\frac{i}{C_3}|\psi_i-\bar\psi_i|\right)\fdA_i\leq C_3,
\end{equation}
where $\bar\psi_i=\frac1{\Imm(z_i)}\int_\Sigma \psi_i\fdA_i$. As for the functions $u_i$, combining Lemma \ref{lemma-aux:uniform_MT} with Claim 2, the Dirichlet energy bound just mentioned, and the lower bound of $\mathrm{IN}(\Sigma,\fg_i)$ shows there is a universal constant $C_4$ so that
\begin{equation}\label{eq-lp:moser-trudinger_for_u}
	\int_\Sigma\exp\left(\frac{i}{C_4}(u_i-\bar u_i)^2\right)\fdA_i\leq C_4,
\end{equation}
where $\bar u_i=\frac1{\Imm(z_i)}\int_\Sigma u_i\fdA_i$. \\

{\noindent\bf{Claim 3:}}
	For any fixed $p\geq1$, we have
	\begin{align}
		& \int_\Sigma e^{p|\psi_i-\bar\psi_i|}\fdA_i\leq (1+o(1))\Imm(z_i), \label{eq-lp:psi_almost_const}\\
		& \int_\Sigma e^{p|u_i-\bar u_i|}\fdA_i\leq (1+o(1))\Imm(z_i), \label{eq-lp:u_almost_const}
	\end{align}
	where $o(1)$ denotes terms which tend to $0$ as $i\to\infty$.
\begin{proof}[Proof of Claim 3]
    We first derive (\ref{eq-lp:psi_almost_const}). Denote $\psi'_i=|\psi_i-\bar\psi_i|$, and set $S_1=\{\psi'_i\leq i^{-1/2}\}$, $S_2=\Sigma\setminus S_1$. By \eqref{eq-lp:brezis-merle_for_psi} we have for $i\gg C_3p$
    \begin{align}
	\int_\Sigma e^{p\psi'_i}\fdA_i 
	&= \int_{S_1}e^{p\psi'_i}\fdA_i+\int_{S_2}e^{p\psi'_i}\fdA_i \\
	&\leq \int_{S_1}e^{p\psi'_i}\fdA_i
		+\int_{S_2}e^{i\psi'_i/C_3}\fdA_i\cdot \exp\big[(p-i/C_3)/\sqrt i\big] \\
	&\leq \Imm(z_i)e^{p/\sqrt i}+C_3\exp\big[p/\sqrt i-\sqrt i/C_3\big] \\
	&= (1+o(1))\Imm(z_i).
    \end{align}
    Analogously for $u_i$, we let $u'_i=|u_i-\bar u_i|$. This time we set $S_1=\{u'_i\leq i^{-1/4}\}$, $S_2=\Sigma\setminus S_1$. By \eqref{eq-lp:moser-trudinger_for_u} and Young's inequality, we have for $i\gg C_4(p/2)^2$
    \begin{align}
	\int_\Sigma e^{pu'_i}\fdA_i &= \int_{S_1} e^{pu'_i}\fdA_i+\int_{S_2} e^{pu'_i}\fdA_i \\
        &\leq \Imm(z_i)e^{pi^{-1/4}}+\int_{S_2}\exp\big(2\sqrt{\frac i{C_4}}u'_i\big)\,\fdA_i \cdot \exp\big[(p-2\sqrt{i/C_4})i^{-1/4}\big] \\
	&\leq \Imm(z_i)e^{pi^{-1/4}}+\int_{S_2}\exp\big(1+\frac i{C_4}(u'_i)^2\big)\,\fdA_i \cdot \exp\big[(p-2\sqrt{i/C_4})i^{-1/4}\big]  \\
	&\leq \Imm(z_i)e^{pi^{-1/4}}+C_4\exp\big[1+p/i^{1/4}-2i^{1/4}/\sqrt C_4\big] \\
	&= (1+o(1))\Imm(z_i).
    \end{align}
\end{proof}

{\noindent \bf{Claim 4:}}	$|\bar u_i|$ and $|\bar\psi_i|$ are uniformly bounded.
\begin{proof}[Proof of Claim 4:]
    Since $x\mapsto e^{2x}$ is convex, we have $e^{2x}\geq e^{2y}+2e^{2y}(x-y)$ for all $x,y\in\RR$. Setting $x=\psi_i$ and $y=-\log\Imm(z_i)/2$, we obtain $e^{2\psi_i}\geq \Imm(z_i)^{-1}+2\Imm(z_i)^{-1}(\psi_i+\frac12\log\Imm(z_i))$. Integrating this,
    \begin{align}\label{eq-lp:psi<}
	&1=|\Sigma|_{\cg_i}=\int_\Sigma e^{2\psi_i}\fdA_i\geq1+2\bar\psi_i+\log\Imm(z_i)
    \end{align}
    which immediately leads to $\bar\psi_i\leq-\frac12\log\Imm(z_i)$. On the other hand, by taking $p=2$ in \eqref{eq-lp:psi_almost_const} we have
    \begin{equation}
        1=\int_\Sigma e^{2\psi_i}\,d\tilde A_i\leq e^{2\bar\psi_i}\cdot(1+o(1))\Imm(z_i).
    \end{equation}
    Combined with the last inequality we have
    \begin{equation}\label{eq-lp:psi_approx}
	\bar\psi_i=-\frac12\log\Imm(z_i)+o(1),
    \end{equation}
    giving the desired bound for $|\bar\psi_i|$ in light of Claim 2 and $\Imm(z_i)\geq\sqrt 3/2$.
    
    For $u_i$ we have
    \begin{align}
	C_1 &\geq |\Sigma|_{\g_i}=\int_\Sigma e^{2\psi_i-2u_i}\fdA_i \\
        &\geq\int_\Sigma\Big[\Imm(z_i)^{-1}+2\Imm(z_i)^{-1}(\psi_i-u_i+\frac12\log\Imm(z_i))\Big]\fdA_i \\
	&\geq 1-o(1)-2\bar u_i.
    \end{align}
    On the other hand, by the minA condition, \eqref{eq-lp:psi_approx}, and Claim 3, we have for sufficiently large $i$
    \begin{align}
        A_0 &\leq \int_\Sigma e^{2\psi_i-2u_i}\,d\tilde A_i\\
        &\leq \Big[\int_\Sigma e^{4\psi_i}\,d\tilde A_i\cdot\int_\Sigma e^{-4u_i}\,d\tilde A_i\Big]^{1/2} \\
        &\leq e^{2\bar\psi_i-2\bar u_i}\cdot \Imm(z_i)(1+o(1))\\
        &\leq 2e^{-2\bar u_i}.
    \end{align}
    This gives the uniform bound on $|\bar u_i|$.
\end{proof}

\begin{proof}[Proof of Theorem \ref{thm-intro:convergence_T3}]
	
    Recall that we have re-written the metrics $\G_i$ as
    \begin{equation}\label{eq-lp:Gi}
        \G_i=e^{2\psi_i-2u_i}\fg_i+e^{2u_i}dt^2,
    \end{equation}
    where $\fg_i$ are the flat metrics on $\RR^2/\text{span}_\mathbb{Z}(1,z_i)$ and $t\in[0,T_i]$. By Claims 1, 2 and 4, we can choose a subsequence and assume that $T_i\to T_\infty$, $z_i\to z_\infty$ and $\bar u_i\to \mathcal{A}$ as $i\to\infty$ for some values $T_\infty, z_\infty, \mathcal{A}\in\RR$. From \eqref{eq-lp:psi_approx} we have $\bar\psi_i\to\mathcal{B}:=-\frac12\log\Imm(z_\infty)$. Let $(M_\infty,\G_\infty)$ be a 3-torus with the metric
    \begin{equation}
        \G_\infty=e^{2\mathcal{B}-2\mathcal{A}}\fg_\infty+e^{2\mathcal{A}}dt^2
    \end{equation}
    where $\fg_\infty$ is the flat metric on $\RR^2/\text{span}_\mathbb{Z}(1,z_\infty)$ and $t\in[0,T_\infty]$. There are natural maps $\Phi_i:(M_\infty,\G_\infty)\to(M_i,\G_i)$ which are the product of the linear maps $\Phi_i:\RR^2/\text{span}_\mathbb{Z}(1,z_\infty)\to\RR^2/\text{span}_\mathbb{Z}(1,z_i)$ and the scalings $t\mapsto\frac{T_i}{T_\infty}t$. Since the underlying flat tori are smoothly converging, we have
    \begin{equation}
        e^{-2|\psi_i-\bar\psi_i|-2|u_i-\bar u_i|}(1-o(1))\G_\infty\leq\Phi_i^*\G_i\leq e^{2|\psi_i-\bar\psi_i|+2|u_i-\bar u_i|}(1+o(1))\G_\infty
    \end{equation}
    where we are implicitly making use of $\Phi_i$ to consider $\psi_i$ and $u_i$ as functions on $M_\infty$.
    
    Let $p_0\geq 1$. To prove the $L^{p_0}$ convergence of metric tensor $\Phi^*\G_i\xrightarrow{L^{p_0}(\G_\infty)}\G_\infty$, it is sufficient to show 
    \begin{equation}
        \int_\Sigma\Big(e^{2p_0|\psi_i-\bar\psi_i|+2p_0|u_i-\bar u_i|}-1\Big)\fdA_i=o(1),
    \end{equation}
    which follows immediately from Claim 3.
    
    Next, for a given $p\in[1,2)$, we will show the convergence $\D_{\G_\infty}(\Phi^*g_i)\xrightarrow{L^p(\G_\infty)}0$. Consider the intermediate metrics $\hat\G_i=\bar g_i+dt^2$, for $t\in[0, T_i]$, on $M_i$. From \eqref{eq-lp:Gi}, we directly estimate
    \begin{align}
    \begin{split}
        \big|\D_{\hat\G_i}\G_i\big|
        &\leq 2\Big(e^{2|\psi_i|}\cdot e^{2|u_i|}\cdot\big(|\fD\psi_i|+|\fD u_i|\big)
        + e^{2|u_i|}\cdot|\fD u_i|\Big) \\
        &\leq 4\,e^{2|\psi_i|}\cdot e^{2|u_i|}\cdot\big(|\fD\psi_i|+|\fD u_i|\big).\label{eq-Lp:nabla1}
    \end{split}
    \end{align}
    Set $q=\left(1+\frac12p\right)$ and notice that $p<q<2$. Integrating \eqref{eq-Lp:nabla1} over $M_i$ and using the general H{\"o}lder's inequality, we have
    \begin{equation}\label{eq-lp:holder}
        ||\D_{\hat\G_i}\G_i||_{L^p(\hat\G_i)}
        \leq 4T_i^{1/p}\Big(||\fD\psi_i||_{L^q}\cdot||e^{2|u_i|}||_{L^r}\cdot||e^{2|\psi_i|}||_{L^r}
        +||\fD u_i||_{L^2}\cdot||e^{2|u_i|}||_{L^s}\cdot||e^{2|\psi_i|}||_{L^s}\Big),
    \end{equation}
    where the powers $r,s>1$ are determined by $\frac2r=\frac1p-\frac1q$, $\frac2s=\frac1p-\frac12$, and all the norms are with respect to $\fg_i$.
    Using Claims 3 and 4 to estimate the $L^r$ and $L^s$ norms in \eqref{eq-lp:holder} and Claim 1 to bound $T_i$, there is a $C_5$ independent of $i$ so that
    \begin{equation}\label{eq-lp:penult}
        ||\D_{\hat\G_i}\G_i||_{L^p(\hat\G_i)}
        \leq C_5(||\fD\psi_i||_{L^q}+||\fD u_i||_{L^2}).
    \end{equation}
    Since Claim 2 asserts that $\mathrm{Im}(z_i)$ is uniformly bounded, we may apply Lemma \ref{lemma-app:w1p} to find constants $C_6, C_7$ independent of $i$ so that
    \begin{equation}\label{eq-lp:dpsi_p}
        ||\fD\psi_i||_{L^q(\fg_i)}\leq C_6||\fDelta\psi_i||_{L^1(\fg_i)}\leq C_7i^{-1},
    \end{equation}
    where we have used \eqref{eq-lp:supersol} and the sentence thereafter to estimate $\overline{\Delta}\psi_i$. Combining \eqref{eq-lp:penult} and \eqref{eq-lp:dpsi_p} with the fact that $\int|\fD u_i|^2\fdA_i=o(1)$, we obtain $||\D_{\hat\G_i}\G_i||_{L^p(\hat\G_i)}\to0$ as $i\to\infty$. Finally, as $\G_\infty$ and $\hat\G_i$ are uniformly bi-Lipschitz equivalent through the maps $\Phi_i$, there is a uniform constant $C_8$ so that
    \begin{equation}
        ||\D_{\G_\infty}(\Phi^*\G_i)||_{L^p(\G_\infty)}\leq C_8||\D_{\hat\G_i}\G_i||_{L^p(\hat\G_i)}\to0
    \end{equation}
    as $i\to\infty$, as claimed.
\end{proof}


\appendix
\section{\texorpdfstring{$L^1$}{L^1} estimates on flat tori}\label{sec:L1_est}
This appendix is aimed at establishing some estimates on the conformal factor $\psi_i$ defined in Section \ref{sec:lp}. The general setup is the following: fix $L>0$ and suppose $\Sigma=\RR^2/\text{span}_\mathbb{Z}(1,z)$ is a flat torus, where $z$ is a complex number such that $\Ree(z)\leq\frac12$, $|z|\geq1$, and $\Imm(z)\leq L$. Note that these conditions imply $\Imm(z)\geq\sqrt3/2$. Let $G(x,y)$ be the Green's function on $\Sigma$, uniquely determined by the conditions
\begin{equation}
    \Delta_x G(x,y)=\delta_y(x)-\frac1{|\Sigma|}, \qquad \int_\Sigma G(x,y)\,dx=0
\end{equation}
for all $x,y\in \Sigma$, where $\Delta_x$ denotes the Laplacian in the $x$-variable and $\delta_y$ denotes the Dirac distribution centered at $y$. Since $z$ lies in a compact region away from $0$ and $1$ in $\mathbb{R}^2$, it is a classical fact that there is a uniform constant $C_1$ depending on $L$ such that
\begin{align}\label{eq-appA:greensbounds}
    \begin{split}
    |G(x,y)|&\leq \frac1{2\pi}|\log d(x,y)|+C_1\\
    |\D G(x,y)|&\leq\frac1{2\pi}\frac1{d(x,y)}+C_1.
    \end{split}
\end{align}
Given a sufficiently regular function $f$ on $\Sigma$ such that $\int_\Sigma f\,dA=0$, the function 
\begin{equation}
    u(x)=\int_\Sigma G(x,y)f(y)\,dy
\end{equation}
solves $\Delta u=f$ and satisfies $\int_\Sigma u\,dA=0$.

\begin{lemma}\label{lemma-app:BM}
    With the assumptions above, there is a constant $C_2$ depending only on $L$ so that the following holds for any $0<\alpha\leq 4\pi$: Given $f\in L^1(\Sigma)$ the function $u(x)=\int_\Sigma G(x,y)f(y)\,dy$ satisfies
    \begin{equation}
    \int_\Sigma\exp\Big(\frac{(4\pi-\alpha)|u|}{||f||_{L^1}}\Big)\,dA\leq C_2\alpha^{-1}.
    \end{equation}
\end{lemma}
\begin{proof}
    Applying Jensen's inequality with the measure $d\mu=\frac{f(y)}{||f||_{L^1}}\,dy$, we can estimate
    \begin{align}
        \int_\Sigma\exp\left(\frac{(4\pi-\alpha)|u(x)|}{||f||_{L^1}}\right)\,dx &\leq \int_\Sigma\exp\left(\int_\Sigma(4\pi-\alpha)|G(x,y)|\cdot\frac{|f(y)|}{||f||_{L^1}}\,dy\right)dx \\
        &\leq \int_\Sigma\left(\int_\Sigma\exp\big((4\pi-\alpha)|G(x,y)|\big)\cdot\frac{|f(y)|}{||f||_{L^1}}\,dy\right)dx \\
        &= \int_\Sigma\frac{|f(y)|}{||f||_{L^1}}\left(\int_\Sigma\exp\big((4\pi-\alpha)|G(x,y)|\big)\,dx\right)\,dy.
    \end{align}
    Then the above integral involving $G(x,y)$ can be estimated using the asymptotics \eqref{eq-appA:greensbounds} to find
    \begin{align}
        \int_\Sigma\exp\left(\frac{(4\pi-\alpha)|u(x)|}{||f||_{L^1}}\right)\,dx &\leq \int_\Sigma\frac{|f(y)|}{||f||_{L^1}}\left(\int_\Sigma e^{4\pi C_1}|d(x,y)|^{-\frac{4\pi-\alpha}{2\pi}}\,dx\right)dy.
    \end{align}
    Since the exponent of $d(x,y)$ satisfies $-\frac{4\pi-\alpha}{2\pi}>-2$, we may decompose $\Sigma=B(x,\frac14)\cup B(x,\frac14)^c$ and compute in radial coordinates 
    \begin{align}
        \int_\Sigma |d(x,y)|^{-\frac{4\pi-\alpha}{2\pi}}\,dx&\leq 2\pi \int^{\frac14}_{0}r^{-1+\frac{\alpha}{2\pi}}\,dr+4|\Sigma|\\
        &\leq 4\pi^2\alpha^{-1}+4L.
    \end{align}
    The result follows.
\end{proof}

\begin{lemma}\label{lemma-app:w1p}
    With the same assumptions as in Lemma \ref{lemma-app:BM}, there is a constant $C_3$ depending only on $L$ so that for any $1\leq p<2$ the function $u(x)=\int_\Sigma G(x,y)f(y)\,dy$ satisfies
    \begin{equation}
        ||\D u||_{L^p}\leq C_3\frac{||f||_{L^1}}{(2-p)^{1/p}}.
    \end{equation}
\end{lemma}
\begin{proof}
    A standard argument using \eqref{eq-appA:greensbounds} allows one to compute the gradient in the $x$ variable by
    \begin{align}
        \D u(x)=\int_\Sigma\D_x G(x,y)f(y)\,dy.
    \end{align}
    Next, we integrate and apply H{\"o}lder's inequality with the measure $d\mu=|f(y)|dy$ to find
    \begin{align}
        \int_\Sigma|\D u(x)|^p\,dx &\leq \int_\Sigma\Big(\int_\Sigma|\D_x G(x,y)|\cdot |f(y)|\,dy\Big)^p\,dx \\
        &\leq \int_\Sigma\Big(\int_\Sigma|\D_x G(x,y)|^p|f(y)|\,dy\Big)\cdot\Big(\int_\Sigma |f(y)|\,dy\Big)^{p-1}\,dx \\
        &\leq ||f||_{L^1}^{p-1}\int_\Sigma\Big(\int_\Sigma C_4\left(1+\frac1{d(x,y)^p}\right)|f(y)|\,dy\Big)\,dx
    \end{align}
    for some constant $C_4$ depending only on $L$, where the last inequality follows from \eqref{eq-appA:greensbounds}. Exchanging the order of integration on the right side yields
    \begin{align}
        \int_\Sigma|\D u(x)|^p\,dx &\leq C_4||f||_{L^1}^{p-1}\int_\Sigma |f(y)|\,dy\int_\Sigma\big(1+\frac1{d(x,y)^p}\big)\,dx \label{eq-appA:singular1}\\
        &\leq C_4||f||_{L^1}^{p-1}\int_\Sigma |f(y)|\,dy\cdot \frac{C_5}{2-p}
    \end{align}
    for some constant $C_5$ depending only on $L$, where we have used $p<2$ when computing the right-most integral in line \eqref{eq-appA:singular1}.
\end{proof}

\bibliographystyle{alpha}
\bibliography{teardropref}

\noindent\textit{Department of Mathematics, Michigan State University,
East Lansing, MI 48824}

\noindent\textit{Email: \href{mailto:kazarasd@msu.edu}{kazarasd@msu.edu}}

\vspace{9pt}

\noindent\textit{Department of Mathematics, Duke University, Durham, NC, 27708,}

\noindent\textit{Email: \href{mailto:kx35@math.duke.edu}{kx35@math.duke.edu}}

\end{document}